\documentclass[11pt,reqno]{amsart}

\usepackage{amsmath}
\usepackage{amsthm}
\usepackage{amsfonts}
\usepackage{amssymb}
\usepackage{color}
\usepackage{graphicx}

\usepackage[margin=1.0in]{geometry}
\usepackage{todonotes}
\usepackage{verbatim}
\usepackage{enumitem}
\numberwithin{equation}{section}

\definecolor{skyblue}{rgb}{0.85,0.85,1}

\newtheorem{lemma}{Lemma}

\newtheorem{theorem}{Theorem}
\newtheorem{cor}{Corollary}

\newtheorem{rem}{Remark}

\newtheorem{define}{Definition}

\newcommand{\bbC}{\mathbb{C}}

\newcommand{\bbR}{\mathbb{R}}

\newcommand{\bbZ}{\mathbb{Z}}
\newcommand{\vp}{\varphi}

\newcommand{\p}{\partial}

\DeclareMathOperator{\Gr}{Gr}

\newcommand{\eps}{\epsilon}

\usepackage[hidelinks]{hyperref}

\begin{document}

\title{A stability index for traveling waves in activator-inhibitor systems}
\author[P. Cornwell]{Paul Cornwell}
\email{pcorn@live.unc.edu}
\address{Department of Mathematics, UNC Chapel Hill, Phillips Hall CB \#3250, Chapel Hill, NC 27516}
\author[C.\,K.\,R.\,T.\ Jones]{Christopher K.\,R.\,T.\ Jones}
\email{ckrtj@email.unc.edu}

\begin{abstract}
We consider the stability of nonlinear traveling waves in a class of activator-inhibitor systems. The eigenvalue equation arising from linearizing about the wave is seen to preserve the manifold of Lagrangian planes for a nonstandard symplectic form. This allows us to define a Maslov index for the wave corresponding to the spatial evolution of the unstable bundle. We formulate the Evans function for the eigenvalue problem and show that the parity of the Maslov index determines the sign of the derivative of the Evans function at the origin. The connection between the Evans function and the Maslov index is established by a ``detection form," which identifies conjugate points for the curve of Lagrangian planes.
\end{abstract}

\maketitle

\tableofcontents

\section{Introduction}
The Maslov index is an integer topological invariant assigned to curves of Lagrangian planes. In recent years, it has been used to study the Morse index of linear operators in differential equations applications, see for example \cite{BCJLMS17,CJM15,HLS16,JLM13}. In this work, we define the Maslov index for traveling waves in activator-inhibitor systems and show how it can be used to give spectral information about the wave. This is accomplished by recasting the Evans function in terms of the symplectic form defining the set of Lagrangian planes. Explicitly, we show that the parity of the Maslov index is the primary factor in determining the sign of the derivative of the Evans function at $\lambda=0$. Employing differential forms, we prove this result entirely using an intersection-based theory of the Maslov index. We introduce a detection form, which is used to identify conjugate points for a curve of Lagrangian spaces. It is the detection form that allows us to bridge the gap from the Evans function to the Maslov index.

The relationship between the Maslov index and the Evans function was discovered by Chardard and Bridges \cite{CB14}. That paper considered evolutionary PDE in which the time-independent part is a finite-dimensional Hamiltonian system in space. Among the systems in this class are reaction-diffusion equations with gradient nonlinearity. In that case, a homoclinic orbit represents a standing wave solution. The stability of such a wave is analyzed by linearizing about the solution and formulating the Evans function $D(\lambda)$ to detect eigenvalues. Since $D(0)=0$ due to translation invariance, the quantity $D'(0)$ provides important information for the stability analysis. Using a symplectic formulation of the Evans function, they were able to show that the sign of $D'(0)$ is determined by the so-called Lazutkin-Treschev invariant, an orientation index for codimension one intersections of Lagrangian planes. They then proved that the sign of this invariant is given by the parity of the Maslov index for the homoclinic orbit, and thus the Maslov index can be used to give the sign of $D'(0)$. This is a very important result, since the Maslov index is a geometric property of the wave itself related to its embedding in phase space. Our goal is to expand this result to a larger class of problems, namely traveling waves in activator-inhibitor systems.

The setting is systems of reaction-diffusion equations of the form \begin{equation}\label{general PDE}
\begin{aligned} u_t & =u_{xx}+f(u)-\sigma v\\
v_t & = v_{xx}+g(v)+\alpha u,
\end{aligned}
\end{equation} where $u,v\in\bbR$, and $x,t\in\bbR$ are space and time respectively. We assume that $f,g\in C^2(\bbR)$, and that $\alpha,\sigma>0$ are real constants. We assume that this system possesses a traveling wave solution, that is, a solution \begin{equation}\vp=(\hat{u},\hat{v})\end{equation} of one variable $z=x-ct$ that decays exponentially as $z\rightarrow\pm\infty$. Without loss of generality, we can take $(\hat{u},\hat{v})\rightarrow(0,0)$ as $z\rightarrow\pm\infty$. We make a few additional assumptions about $\vp(z)$ at the beginning of \S 2.

One of the motivations for studying systems of this form is the so-called ``Turing problem,'' concerning pattern formation in reaction-diffusion equations. In his classic paper \cite{turing}, Turing showed that the equilibrium $(u,v)=(0,0)$ can be stable in the local reaction, but unstable when diffusion is added to the equation. This, in turn, can lead to the formation of patterns and has been invoked to explain patterns appearing in nature, such as stripes on a zebra \cite{Murray}. However, this will happen only under certain circumstances. First, the system must be of \emph{activator-inhibitor} type, meaning that the Jacobian of the nonlinearity evaluated at $(0,0)$ must have opposite signs on both its diagonal and off-diagonal terms. Also, the diffusivities of the two agents $u$ and $v$ must be significantly different. By taking the diffusion coefficients to be the same in (\ref{general PDE}) we are therefore investigating a regime in which Turing bifurcations will not occur. 

In light of the preceding discussion, we assume that $f$ and $g$ are such that the constant solution is stable without diffusion. Writing the nonlinearity in (\ref{general PDE}) as $F:\bbR^2\rightarrow \bbR^2$, we see that the explicit conditions are determined by the trace and determinant of $DF(0)$: \begin{equation}\label{Turing assumptions}
\begin{aligned}
\mathrm{Trace}(DF(0)) & = f'(0)+g'(0)<0\\
\mathrm{Det}(DF(0)) & = f'(0)g'(0)+\sigma\alpha>0.
\end{aligned}
\end{equation}
With the stated sign conventions on $\sigma$ and $\alpha$, we now see that (\ref{general PDE}) is an activator-inhibitor system, susceptible to diffusion-driven instability if the diffusivities were changed. (The signs of $f'(0)$ and $g'(0)$ do not affect what follows.) A fundamental issue in the study of reaction-diffusion equations is how the interplay between diffusion and the nonlinearity can create patterns, traveling waves, and other coherent structures. Equally important is the question of the stability of such structures, and the geometric techniques employed herein aim to shed light on this phenomenon.

For equations of the form (\ref{general PDE}), it is well known \cite{BJ89,Henry} that proving ${\vp}$ is stable is tantamount to showing that the spectrum of the operator $L$ arising from linearizing the traveling wave equation about ${\vp}$ is bounded away from the imaginary axis in the left half-plane. One can then set up the eigenvalue equation as a matrix system \begin{equation}\label{matrix system}
Y'(z)=A(\lambda,z)Y(z),\hspace{.1 in}\left(^\prime=\frac{d}{dz}\right)
\end{equation} which allows for the use of geometric dynamical systems techniques to locate eigenvalues. The classic tool in this pursuit is the Evans function, which is a Wronskian-type determinant that detects linear dependence between solutions that decay in backwards time with those that decay in forwards time. We call these (two-dimensional) solution spaces the unstable bundle $E^u(\lambda,z)$ and the stable bundle $E^s(\lambda,z)$ respectively. Via the Pl\"{u}cker embedding (see \S 3), $E^{u/s}(\lambda,z)$ can be related to elements of the second exterior power of $\bbC^4$, which we call $\tilde{E}^{s/u}(\lambda,z)$. As is done in \cite{AGJ}, the Evans function $D(\lambda)$ is then defined as the wedge product of $\tilde{E}^u(\lambda,z)$ and $\tilde{E}^s(\lambda,z)$. It follows that $D(\lambda)=0$ if and only $\lambda$ is an eigenvalue of $L$. Additionally, the Evans function is known to satisfy the following. \begin{enumerate}
\item $D$ is analytic on an open domain containing the closed right half-plane.
\item $D(\lambda)\in\bbR$ if $\lambda\in\bbR$.
\item The order of $\lambda$ as a root of $D$ equals the algebraic multiplicity of $\lambda$ as an eigenvalue of $L$.
\end{enumerate}
The reader is referred to \S 4.1 of \cite{Sandstede02} or \S9.1-9.3 of \cite{KP13} for an overview of $D(\lambda)$ and its properties.

The key insight in this paper is that there is a symplectic form $\omega$ whose value on solutions can easily be tracked. In particular, the set of Lagrangian planes with respect to this form is invariant. We will show that the space $E^{u}(\lambda,z)$ is $\omega-$Lagrangian for any $\lambda\in\bbR$ and $z\in\bbR$, and hence for fixed $\lambda$, the spatial evolution of $E^u(\lambda,z)$ defines a curve in this set. The set of $\omega-$Lagrangian planes is actually a manifold with cyclic fundamental group, and hence we can assign an index to curves in this space. This is the Maslov index, which will be defined for the traveling wave in \S 5. 

As observed in \cite{BD01,BD03,CB14}, the Evans function can be recast as a ``symplectic determinant'' through the use of a natural volume form on a four-dimensional vector space. In \S 4, we give the corresponding formula for the activator-inhibitor case. This formula will allow us to connect the Maslov index to the Evans function, and hence to the stability of ${\vp}$. Due to translation invariance, it is immediate that $D(0)=0$. For $\lambda\in\bbR$ large, it can also be shown that $D(\lambda)>0$. It follows that the sign of $D'(0)$ can be used as an instability index, or as a stability index if other information about $\sigma(L)$ is known (e.g. \cite{Jones84}). This parity argument has been used many times in the stability analysis of nonlinear waves, for example \cite{AJ94,BD01,PW,GZ98}. Investigating the quantity $D'(0)$ is therefore worthwhile as a means for discovering new mechanisms of instability.

Now focusing on the $\lambda=0$ case, we know that the unstable and stable bundles are spanned by the derivative of the wave $\vp'$ and one other vector each. Suitably normalized, we have then that $E^s(0,z)=\text{sp}\{u_1(z),\vp'\}$ and $E^u(0,z)=\text{sp}\{\vp',u_4(z)\}$. In \S 4, we prove the primary result of this paper: \begin{equation}\label{derivative intro}
D'(0)=e^{cz}\omega(u_1,u_4)\int\limits_{-\infty}^{\infty}e^{cz}\left(\frac{(\hat{u}')^2}{\sigma}-\frac{(\hat{v}')^2}{\alpha}\right)\,dz.
\end{equation} The expression $e^{cz}\omega(u_1,u_4)$, which is actually independent of $z$, is referred to in \cite{CB14} as the \emph{Lazutkin-Treschev} invariant of the wave. The integral in (\ref{derivative intro}) can be calculated numerically if the wave is known, but the Lazutkin-Treschev invariant is more elusive. Briefly, the exponential weight can be distributed so that the solution $u_1$ approaches an eigenvector of the asymptotic matrix $A_\infty(0)$ as $z\rightarrow\infty$. Along with the $z$-independence of $e^{cz}\omega(u_1,u_4)$, this can then be used to argue that $u_4$ (with the remaining weight) approaches another eigenvector of $A_\infty(0)$. However, the orientation of this latter eigenvector is unclear, which is vital in determining the sign of the symplectic form.

It turns out that the Maslov index, which measures the winding of the unstable bundle as it moves along the orbit, is the key to determining the sign of the Lazutkin-Treschev invariant. We will show in \S 5 that \begin{equation}\label{parity result intro}
(-1)^{\mathrm{Maslov}({\vp})}=\text{sign}\left(e^{cz}\omega(u_1,u_4)\right).
\end{equation}
This is the analog of the result proved in \cite{CB14}. The difficulty in treating activator-inhibitor systems lies in the fact that the operator obtained from linearizing about the wave solution is not self-adjoint because the nonlinearity in (\ref{general PDE}) is not a gradient. As a consequence, its spectrum is not necessarily real, and the standard symplectic form is not preserved on solutions to the eigenvalue problem. Furthermore, we have a second obstruction to self-adjointness in the speed $c$. A major part of the analysis lies in determining the role of this parameter.

The topology of the set of Lagrangian planes-the \emph{Lagrangian Grassmannian} $\Lambda(2)$-is understood by realizing $\Lambda(2)$ as the homogeneous space $U(2)/O(2)$. In particular, the Maslov index can be related to the spectral flow of a family of unitary matrices representing a curve of subspaces. This approach is preferred in differential equations applications \cite{CH,CH14,HLS16,HS16,JLS17}, in which the Maslov index is used to count unstable eigenvalues for Schr\"{o}dinger operators. In \cite{CB14} a variant of this approach due to Souriau is used in which the Maslov index is defined for elements in the universal cover of $U(n)$. There is an impediment to taking this perspective for activator-inhibitor systems owing to the fact that a non-standard symplectic form is used. To encode the Lagrangian planes in unitary matrices would require a change of variables in the eigenvalue equation, which would complicate calculating the Maslov index in practice.

Our approach instead uses an entirely intersection-based theory of the Maslov index, without any reference to unitary matrices. Using differential forms, we define a function $\beta$ whose zeros give the dimension of crossings of $E^u(0,z)$ with the train of the stable subspace of the rest state. Derivatives of $\beta$ can be related to the crossing form for intersections of curves of Lagrangian planes, which in turn is used to define the Maslov index. In addition to appealing directly to the intuitive idea of curves crossing hypersurfaces, the use of differential forms in this way is also instrumental in calculating the Maslov index in examples. 

To make use of (\ref{parity result intro}), one would need to calculate the Maslov index of the traveling wave. Efficient ways of doing so numerically were developed in \cite{BM15,CDB09}. Alternatively, one can use the crucial fact that, when $\lambda=0$, the unstable bundle is tangent to the unstable manifold $W^u(0)$ of the rest state along $\vp$. With this pretty interpretation, (\ref{derivative intro}) and (\ref{parity result intro}) can be paraphrased by saying that the derivative of the Evans function at $\lambda=0$ is determined by how many times the unstable manifold twists as it moves along the wave. This observation yields a practical way of calculating $\mathrm{Maslov}(\vp)$ if one has a way of tracking invariant manifolds in the nonlinear system. For example, if there is a timescale separation in the traveling wave equation, then the techniques of geometric singular perturbation \cite{Fen79,JoGSP} can be used to do just that. In \S 6, we outline how the Maslov index can be calculated in a doubly-diffusive FitzHugh-Nagumo system. In contrast with orientation indices using invariant manifolds in \cite{AJ94,Jones84}, we \emph{do not} use derivatives with respect to $c$ to achieve our result.
\section{The Eigenvalue Problem and Evans Function}
Recasting (\ref{general PDE}) in a moving frame, one sees that a traveling wave is a steady state of \begin{equation}\label{traveling wave PDE}
\begin{aligned} u_t&=u_{zz}+cu_z+f(u)-\sigma v\\
v_t&=v_{zz}+cv_z+g(v)+\alpha u.
\end{aligned}
\end{equation} The steady state equation is an ODE, which upon setting $u_z=\sigma w$ and $v_z=\alpha y$, can be written as the first order system \begin{equation}\label{traveling wave ODE}
\left(\begin{array}{c}
u\\
v\\
w\\
y
\end{array}\right)_z=\left(\begin{array}{c}
\sigma w\\
\alpha y\\
-cw-\frac{f(u)}{\sigma}+v\\
-cy-u-\frac{g(v)}{\alpha}
\end{array}\right).
\end{equation}
In an abuse of notation, we will write $\varphi(z)$ for the solution of interest both of (\ref{traveling wave PDE}) (i.e. $\varphi=(\hat{u},\hat{v})$) and of (\ref{traveling wave ODE}) (i.e. $\varphi=(\hat{u},\hat{v},\hat{u}'/\sigma,\hat{v}'/\alpha)$). It will be clear from context which object is being referenced. Since $\hat{u},\hat{v}$ decay exponentially, one sees that $\varphi$ is a homoclinic orbit to $0=(0,0,0,0)$ for (\ref{traveling wave ODE}). The linearization of (\ref{traveling wave ODE}) about this fixed point is given by \begin{equation}\label{linearization about 0}
\left(\begin{array}{c}
\delta u\\ \delta v\\ \delta w\\ \delta y
\end{array}\right)=\left(\begin{array}{c c c c }
0 & 0 & \sigma & 0\\
0 & 0 & 0 & \alpha\\
-\frac{f'(0)}{\sigma} & 1 & -c & 0\\
-1 & -\frac{g'(0)}{\alpha} & 0 & -c
\end{array}\right)\left(\begin{array}{c}
\delta u\\ \delta v\\ \delta w\\ \delta y
\end{array}\right).
\end{equation} Let $\nu_{1,2}$ be the eigenvalues of $DF(0)$--both of which have negative real part in light of (\ref{Turing assumptions}). Writing the matrix in (\ref{linearization about 0}) in block form, it is a simple calculation to see that its eigenvalues are given by \begin{equation}\label{evals of 0}
\begin{aligned}
\mu_1(0) & = -\frac{c}{2}-\frac{1}{2}\sqrt{c^2-4\nu_1}\\
\mu_2(0) & = -\frac{c}{2}-\frac{1}{2}\sqrt{c^2-4\nu_2}\\
\mu_3(0) & = -\frac{c}{2}+\frac{1}{2}\sqrt{c^2-4\nu_2}\\
\mu_4(0) & = -\frac{c}{2}+\frac{1}{2}\sqrt{c^2-4\nu_1}.
\end{aligned}
\end{equation} The argument $0$ is included in anticipation of extending this calculation to other values of the spectral parameter $\lambda$. It is clear that $\text{Re } \mu_{1,2}(0)<0$ and that $\text{Re } \mu_{3,4}(0)>0$, hence $0$ is a hyperbolic fixed point with two-dimensional stable and unstable manifolds, $W^s(0)$ and $W^u(0)$. We denote by $V^{s/u}(0)$ the stable and unstable subspaces of the linearization (\ref{linearization about 0}). Thus \begin{equation}
V^s(0)=T_0W^s(0),\hspace{.2 in} V^u(0)=T_0 W^u(0).
\end{equation} Our assumptions then guarantee that $\varphi$ lies in the intersection of $W^s(0)$ and $W^u(0)$. We make the following additional assumptions about $\varphi$. \begin{enumerate}[label=(A\arabic*)]
\item $c<0$. That is, the wave propagates to the left.
\vspace{.1 in}
\item The tails of $\varphi$ are monotone, as opposed to oscillatory. From (\ref{evals of 0}), we see that this is equivalent to assuming that $\nu_1$ and $\nu_2$ are real. Additionally, we assume that $\nu_1$ and $\nu_2$ are simple and satisfy \begin{equation}
\nu_1<\nu_2<0.
\end{equation}
\item Assumption (A2) guarantees that $\mu_1(0)<\mu_2(0)<0<\mu_3(0)<\mu_4(0)$, so that the leading eigenvalues $\mu_2(0)$ and $\mu_3(0)$ are real and simple. We assume that the exponential decay rate of $\varphi$ (as a homoclinic orbit) is given by $\mu_2(0)$ in forwards time and by $\mu_3(0)$ in backwards time. This assumption is generic, c.f. \S 2.1 of \cite{HomSan10}.
\vspace{.1 in}
\item $\varphi$ is \emph{transversely constructed}. This means that $(\varphi(z),c)\in\bbR^5$ is given by the transverse intersection of the center-unstable and center-stable manifolds of the fixed point $(0,0,0,0,c)$ for (\ref{traveling wave ODE}) with the equation $c'=0$ appended.
\end{enumerate}

Assumption (A1) is not essential; all of the arguments of this paper go through \emph{mutatis mutandis} if the speed is taken instead to be positive. The reason for assumptions (A2) and (A3) is that some of our proofs (e.g. Lemma \ref{Lemma D prime}) use the decay properties of the wave. Our results could be extended to the case where $\varphi$ has oscillatory tails by replacing $\mu_i$ with $\text{Re } \mu_i$ in most proofs. In this case, $\nu_1$ and $\nu_2$ are complex conjugates, and consequently the stable and unstable eigenvalues of (\ref{linearization about 0}) come in conjugate pairs as well. Likewise, if $\varphi$ is in an orbit-flip configuration \cite{HomSan10}, then the proofs could be altered accordingly, provided that the decay properties are known. Thus assumptions (A1)-(A3) are mostly for notational convenience. However, if $\nu_1=\nu_2$, then we have $\mu_1(0)=\mu_2(0)$ and $\mu_3(0)=\mu_4(0)$. This causes trouble for analytically choosing bases for $E^{u/s}(\lambda,z)$ \cite{BD03}, so we assume that this is not the case. We remark that the simplicity of the eigenvalues $\mu_i(0)$ is generic vis-\`{a}-vis $\sigma$ and $\alpha$. Finally, we impose (A4) because $D'(0)$ is known to vanish if the wave is \emph{not} constructed in this manner, c.f. pp. 57-60 of \cite{AJ94}. This assumption is therefore natural if one wishes to use the Maslov index to say something about $\mathrm{sign}\,D'(0)$.

With the assumptions about $\varphi$ in place, we turn to the stability question. As explained in \S 2.A of \cite{AGJ}, equation (\ref{traveling wave PDE}) can be solved for small $t$ in the space $BU(\bbR,\bbR^2)$ of bounded, uniformly continuous functions on $\bbR$, so this is a natural space to consider for the stability analysis. In that work, assumptions are also made on the decay rates of the wave in its tails. Those assumptions are satisfied here, since $0$ is a hyperbolic equilibrium for (\ref{traveling wave ODE}), as was shown above. We use the following definition for nonlinear stability. \begin{define}
The traveling wave $\vp(z)$ is \textbf{asymptotically stable} relative to (\ref{traveling wave PDE}) if there is a neighborhood $V\subset BU(\bbR,\bbR^2)$ of $\vp(z)$ such that if $u(z,t)$ solves (\ref{traveling wave PDE}) with $u(z,0)\in V$, then \[||\vp(z+k)-u(z,t)||_\infty\rightarrow 0 \] as $t\rightarrow\infty$ for some $k\in\bbR$.
\end{define} As advertised, it suffices to prove that the wave is spectrally stable. We linearize (\ref{traveling wave PDE}) about the traveling wave $\vp$ to obtain \begin{equation}\label{defn of L}
P_t=P_{zz}+cP_z+\left(\begin{array}{c c}
f'(\hat{u}) & -\sigma\\
\alpha & g'(\hat{v})
\end{array}\right)P,
\end{equation} with $P=(p,q)^T\in BU(\bbR,\bbR^2)$. The right-hand-side of (\ref{defn of L}) defines the operator $L$ through its action on $P$. The spectrum $\sigma(L)$ of $L$ is then broken into two parts. First, $\lambda\in\bbC$ is an eigenvalue for $L$ if there exists a bounded solution $P$ to the equation \begin{equation}\label{eval equation}
LP=\lambda P.
\end{equation} The set of eigenvalues of finite multiplicity is called the point spectrum $\sigma_n(L)$ of $L$, and the complement of this set in $\sigma(L)$ is called the essential spectrum $\sigma_\text{ess}(L)$. Equation (\ref{eval equation}) is a second-order system of two equations, which can be converted to a four-dimensional first-order system using the standard trick. This allows us to analyze the eigenvalue problem using geometric dynamical systems methods. To that end, set $p_z=\sigma r$ and $q_z=\alpha s$ to get \begin{equation}\label{eval eqn matrix}
\left(\begin{array}{c}
p\\
q\\
r\\
s
\end{array}\right)_z=\left(\begin{array}{c c c c}
0 & 0 & \sigma & 0\\
0 & 0 & 0 & \alpha\\
\frac{\lambda}{\sigma}-\frac{f'(\hat{u})}{\sigma} & 1 & -c & 0\\
-1 & \frac{\lambda}{\alpha}-\frac{g'(\hat{v})}{\alpha} & 0 & -c
\end{array}\right)\left(\begin{array}{c}
p\\
q\\
r\\
s
\end{array}\right).
\end{equation} This is now in the familiar form \begin{equation}\label{eval eqn 1st order}
Y'=A(\lambda,z)Y,
\end{equation} where $Y(\lambda,z)=(p,q,r,s)\in\bbC^4$ and $A(\lambda,z)\in M_4(\bbC)$ is the complex matrix \begin{equation}\label{matrix A defn}
A(\lambda,z)=\left(\begin{array}{c c c c}
0 & 0 & \sigma & 0\\
0 & 0 & 0 & \alpha\\
\frac{\lambda}{\sigma}-\frac{f'(\hat{u})}{\sigma} & 1 & -c & 0\\
-1 & \frac{\lambda}{\alpha}-\frac{g'(\hat{v})}{\alpha} & 0 & -c
\end{array}\right).
\end{equation} Since $(\hat{u},\hat{v})\rightarrow(0,0)$ exponentially as $z\rightarrow\pm\infty$, it follows that \begin{equation} A(\lambda,z)\rightarrow A_\infty(\lambda)=\left(\begin{array}{c c c c}
0 & 0 & \sigma & 0\\
0 & 0 & 0 & \alpha\\
\frac{\lambda}{\sigma}-\frac{a}{\sigma} & 1 & -c & 0\\
-1 & \frac{\lambda}{\alpha}-\frac{b}{\alpha} & 0 & -c
\end{array}\right),\end{equation} where $a=f'(0)$ and $b=g'(0)$, and this convergence is exponential as well. Observe that one obtains (\ref{linearization about 0}) by setting $\lambda=0$ in $A_\infty(\lambda)$. More generally, (\ref{eval eqn 1st order}) with $\lambda=0$ is the variational equation for (\ref{traveling wave ODE}) along $\varphi$. This is a major motivation for studying the Maslov index in this context. We will elaborate on this in \S 6.

The asymptotic matrix $A_\infty(\lambda)$ plays a prominent roll in the behavior of solutions to (\ref{eval eqn 1st order}). In particular, if $A_\infty(\lambda)$ has no purely imaginary eigenvalues, then (\ref{eval eqn 1st order}) admits an exponential dichotomy \cite{Sandstede02}, and hence we can pick out solutions that decay at either plus or minus infinity. The essential spectrum of $L$ is precisely the set of $\lambda$ for which $A_\infty(\lambda)$ has eigenvalues of zero real part, c.f. Lemma 3.1.10 of \cite{KP13}. Fortunately, in the activator-inhibitor case this set is bounded away from the imaginary axis in the left half-plane.
\begin{lemma}\label{essential spectrum}
There exists $K<0$ such that for the operator $L$ defined by (\ref{defn of L}), we have \begin{equation}
\sigma_\text{ess}(L)\subset \mathcal{K}:=\{z\in\bbC:\text{Re }z\leq K \}.
\end{equation} Furthermore, $A_\infty(\lambda)$ has exactly two eigenvalues of positive real part and two of negative real part for all $\lambda\in\bbC\setminus \mathcal{K}.$
\end{lemma}
\begin{proof}
Since $L$ is known to be sectorial \cite{Henry}, it suffices to show that $\sigma_{\text{ess}}$ is disjoint from the closed right half-plane. For fixed $\lambda\in\bbC$, the characteristic polynomial of $A_\infty(\lambda)$ is \begin{equation}\label{char poly}
\chi_{A_\infty(\lambda)}(t)=t^4+2ct^3+(c^2+a+b-2\lambda)t^2+c(a+b-2\lambda)t+ab-\lambda(a+b)+\lambda^2+\sigma\alpha.
\end{equation}
Set $\lambda=x+iy$. Then $\lambda\in\sigma_{\text{ess}(L)}$ if and only if $A_\infty(\lambda)$ has at least one purely imaginary eigenvalue, which is to say that $\chi_{A_\infty(\lambda)}(ik)=0$ for some $k\in\bbR$. We substitute these values into $\chi_{A_\infty(\lambda)}$ and collect the real and imaginary parts: \begin{equation}\label{char poly reim}
\begin{aligned}
\text{Re}\,\chi_{A_\infty(\lambda)}(ik) & =\left(x+k^2-\left(\frac{a+b}{2}\right)\right)^2-(y-ck)^2-\frac{1}{4}\left(a-b\right)^2+\sigma\alpha\\
\text{Im}\,\chi_{A_\infty(\lambda)}(ik) & =(ck-y)(-2k^2+a+b-2x).
\end{aligned}
\end{equation}
From the first equation in (\ref{Turing assumptions}), it follows that for $x\geq0$, the second equation in (\ref{char poly reim}) vanishes only when $y=ck$. Substituting this value into the first equation, we end up with \begin{equation}
\text{Re}\chi_{A_\infty(x+ick)}(ik)=k^4-(a+b-2x)k^2+x^2-(a+b)x+\alpha\sigma+ab.
\end{equation} This even quartic has a real solution $k$ for nonnegative $x$ if and only if the constant term is nonpositive. But the constant term is a quadratic in $x$, which is positive for $x\geq 0$, by (\ref{Turing assumptions}). The second claim follows from examining (\ref{matrix A defn}) for real $\lambda\gg 1$ and noticing that changes in the eigenvalue split occur only when $\lambda$ crosses $\sigma_\text{ess}$.
\end{proof}

The preceding lemma guarantees that we can define the Evans function \cite{AGJ,Jones84,KP13,Sandstede02} in an open, simply connected domain $\mathcal{U}\subset\bbC$ containing the closed right half-plane $\bar{\bbC}^+$. Before doing so, consider $\lambda\in\bbR$. An analogous calculation to (\ref{evals of 0}) shows that the eigenvalues of $A_\infty(\lambda)$ are \begin{equation}\label{evals of A(lambda)}
\begin{aligned}
\mu_1(\lambda) & = -\frac{c}{2}-\frac{1}{2}\sqrt{c^2+4(\lambda-\nu_1)}\\
\mu_2(\lambda) & = -\frac{c}{2}-\frac{1}{2}\sqrt{c^2+4(\lambda-\nu_2)}\\
\mu_3(\lambda) & = -\frac{c}{2}+\frac{1}{2}\sqrt{c^2+4(\lambda-\nu_2)}\\
\mu_4(\lambda) & = -\frac{c}{2}+\frac{1}{2}\sqrt{c^2+4(\lambda-\nu_1)}.
\end{aligned}
\end{equation} As long as $\lambda>\nu_2$, the $\mu_i$ are all real and simple, by (A2). Let $\delta>0$ be so that this is the case on the interval \begin{equation}\label{I defn}
I=[-\delta,\infty)\subset\bbR\cap\mathcal{U}.
\end{equation} It is then known (e.g. page 56 of \cite{PW}) that there exist solutions $u_i(\lambda,z)$ to (\ref{eval eqn 1st order}) satisfying \begin{equation}\label{individual soln decay}
\begin{aligned}
\lim\limits_{z\rightarrow\infty}e^{-\mu_i(\lambda)z}u_i(\lambda,z) & =\eta_i(\lambda), \hspace{.2 in}i=1,2\\
\lim\limits_{z\rightarrow-\infty}e^{-\mu_i(\lambda)z}u_i(\lambda,z) & = \eta_i(\lambda), \hspace{.2 in}i=3,4,
\end{aligned}
\end{equation} where $\eta_i(\lambda)$ is a nonzero eigenvector of $A_\infty(\lambda)$ corresponding to eigenvalue $\mu_i(\lambda)$. Furthermore, the $u_i$ are analytic in $\lambda$, and the limits are achieved uniformly on compact subsets of $I$. Following \S 3 of \cite{AGJ}, we then define $\lambda$- and $z$-dependent real vector spaces \begin{equation}\label{un/stable bundles}
\begin{aligned}
E^s(\lambda,z) & =\mathrm{sp}\{u_1(\lambda,z),u_2(\lambda,z)\}\\
E^u(\lambda,z) & =\mathrm{sp}\{u_3(\lambda,z),u_4(\lambda,z)\}.
\end{aligned}
\end{equation} We call $E^s$ the \emph{stable bundle} and $E^u$ the \emph{unstable bundle}. These sets could be defined for complex $\lambda\in\mathcal{U}$ as well, but this work is focused on $\lambda=0$, so it is unnecessary to do so. Instead, we will define the Evans function on the rest of $\mathcal{U}$ using the exterior powers $\bigwedge^k\bbC^4$ of $\bbC^4$. 

System (\ref{eval eqn 1st order}) induces an equation \begin{equation}\label{derivation on ext alg}Z'=A^{(2)}(\lambda,z)Z,
\end{equation} on $\bigwedge^2\bbC^4$. Explicitly, $A^{(2)}(\lambda,z)$ is the unique endomorphism of $\bigwedge^2\bbC^4$ satisfying \begin{equation}\label{induced eqn specific} A^{(2)}(\lambda,z)(Y_1\wedge Y_2)=A(\lambda,z)Y_1\wedge Y_2+Y_1\wedge A(\lambda,z)Y_2\end{equation} for any $Y_{1,2}\in \bbC^4$. It follows that if $Y_i$ are solutions to (\ref{eval eqn 1st order}), then $Y_1\wedge Y_2$ is a solution to (\ref{derivation on ext alg}). The eigenvalues of the matrix $A_\infty^{(2)}(\lambda)=\displaystyle\lim\limits_{z\rightarrow\pm}A^{(2)}(\lambda,z)$ are the pairwise sums of eigenvalues of $A_\infty(\lambda)$, so for any $\lambda\in\mathcal{U}$ we have a simple eigenvalue of largest (positive) real part and a simple eigenvalue of least (negative) real part, by Lemma \ref{essential spectrum}. More specifically, $\mu_3(\lambda)+\mu_4(\lambda)$ and $\mu_1(\lambda)+\mu_2(\lambda)$ are the eigenvalues of $A_\infty^{(2)}(\lambda)$ of largest and smallest real part respectively. Let $\zeta_s(\lambda)\in\bigwedge^2\bbC^4$ and $\zeta_u(\lambda)\in\bigwedge^2\bbC^4$ be eigenvectors corresponding to $\mu_1(\lambda)+\mu_2(\lambda)$ and $\mu_3(\lambda)+\mu_4(\lambda)$ respectively. It follows from \S II.4.2 of \cite{Kato} that $\zeta_{s/u}(\lambda)$ can be chosen analytically in $\lambda$.

As in \S 4 of \cite{AGJ}, we have solutions $\tilde{E}^s(\lambda,z)$ and $\tilde{E}^u(\lambda,z)$  to (\ref{derivation on ext alg}) satisfying \begin{equation}\label{soln decay}
\begin{aligned}
e^{-\left(\mu_1(\lambda)+\mu_2(\lambda)\right)z}\tilde{E}^s(\lambda,z)\rightarrow\zeta_s(\lambda) & \text{ as } z\rightarrow\infty\\
e^{-\left(\mu_3(\lambda)+\mu_4(\lambda) \right)z}\tilde{E}^u(\lambda,z)\rightarrow\zeta_u(\lambda) & \text{ as } z\rightarrow-\infty,
\end{aligned}
\end{equation} and these are unique up to scalar multiplication. We shall refer to $\tilde{E}^s(\lambda,z)$ and $\tilde{E}^u(\lambda,z)$ as the stable and unstable two-vectors respectively. The wedge product $\tilde{E}^s(\lambda,z)\wedge \tilde{E}^u(\lambda,z)$ is in the one-dimensional space $\bigwedge^4\bbC^4$, on which (\ref{eval eqn 1st order}) induces the equation \begin{equation}
W'=\text{Trace}(A(\lambda,z))W,
\end{equation} see \S 4 of \cite{AGJ}. Furthermore, it is shown in \cite{AGJ} that $\tilde{E}^{s}(\lambda,z)\wedge\tilde{E}^u(\lambda,z)$ solves this equation. Since $\text{Trace}(A(\lambda,z))\equiv -2c$, it follows that \begin{equation}\label{Evans fct wedge}
\tilde{D}(\lambda)=e^{2cz}\tilde{E}^s(\lambda,z)\wedge \tilde{E}^u(\lambda,z)
\end{equation}is independent of $z$. This is the Evans function, as defined in \cite{AGJ}. Since $\bigwedge^4\bbC^4$ is only one-dimensional, we know that $\tilde{D}(\lambda)$ is a scalar multiple of the volume element \begin{equation}
\mathrm{vol}^*=e_1\wedge e_2\wedge e_3\wedge e_4,
\end{equation} where $\{e_i\}_{i=1}^4$ is the standard basis of $\bbC^4$. For matters of convenience, we define the Evans function to be the scalar part of the wedge product in (\ref{Evans fct wedge}). \begin{define}
The \textbf{Evans function} $D(\lambda)$ is defined by \begin{equation}\label{Evans fct defn}
D(\lambda)\mathrm{vol}^*=e^{2cz}\tilde{E}^s(\lambda,z)\wedge \tilde{E}^u(\lambda,z).
\end{equation} 
\end{define}
This function has the properties of the Evans function outlined in the introduction. In particular, it is analytic on $\mathcal{U}$ and vanishes for all values of $\lambda$ that are eigenvalues of $L$, with the order of the zero giving the algebraic multiplicity of the eigenvalue. The reason that the Evans function detects eigenvalues is that the two-vectors $\tilde{E}^s(\lambda,z)$ and $\tilde{E}^u(\lambda,z)$ capture--via (\ref{derivation on ext alg})--the solutions to (\ref{eval eqn 1st order}) that decay in forwards and backwards time respectively. The key to making this rigorous is the Pl\"{u}cker embedding, which we discuss in the next section.

We conclude this section by pointing out that $D(\lambda)$ is real-valued for $\lambda\in\bbR$, since $\tilde{E}^{s/u}(\lambda,z)$ are solutions of an ODE with real coefficients in that case. Furthermore, since eigenvectors are only defined up to a scalar multiple, we fix an orientation by demanding that \begin{equation}\label{orientation}
\zeta_s(0)\wedge\zeta_u(0)=\rho\,\mathrm{vol}^*,\hspace{.2 in}\rho>0.
\end{equation} Under this assumption, it is known (c.f. Lemma 4.2 of \cite{Yan02} and Lemma 4.2 of \cite{AJ94}) that \begin{equation}
D(\lambda)>0 \text{ for } \lambda\gg 1.
\end{equation}
\section{Symplectic Structure and the Pl\"{u}cker Coordinates}
In the previous section, the stable and unstable two-vectors were defined as elements in the 2nd exterior power $\bigwedge^2\bbC^4$. Via the Pl\"{u}cker embedding, they can also be related to the stable and unstable bundles defined by (\ref{un/stable bundles}). For the rest of the paper, we focus exclusively on $\lambda\in\bbR$. Let $\{e_i\}_{i=1}^4$ be the standard basis for $\bbR^4$. To any plane $V=\text{sp}\{u,v\}\in\text{Gr}_2(\bbR^4)$ ($u=\sum_{i=1}^{4}u_ie_i$, e.g.) we can associate a two vector $u\wedge v$, which can be expressed in terms of the basis $\{e_i\wedge e_j\}$ of $\bigwedge^2\bbR^4$. Explicitly, the $e_i\wedge e_j-$coordinate--call it $p_{ij}$--is given by \begin{equation}\label{Plucker coord defn}
p_{ij}=\left|\begin{array}{c c}
u_i & v_i\\
u_j & v_j
\end{array}\right|.
\end{equation} 
Choosing a different basis of $V$ would change this two-vector by a non-zero constant multiple, so by projectivizing, one obtains a well-defined map $j:\text{Gr}_2(\bbR^4)\rightarrow\mathbb{P}(\bigwedge^2\bbR^4)$. It can be shown \cite{Hassett} that this map is an embedding, and its image is a projective variety. For any plane $V$, we call $j(V)=(p_{12},p_{13},p_{14},p_{23},p_{24},p_{34})$ the \emph{Pl\"{u}cker coordinates} for $V$. These coordinates are really projective, but it will nonetheless be useful to think of a plane as a vector in $\mathbb{R}^6$. Finally, it is important to note that the map $j$ is not surjective--it is explained in \cite{Hassett} that only points satisfying the \emph{Grassmannian condition} \begin{equation}\label{Grassmannian condition}
p_{12}p_{34}-p_{13}p_{24}+p_{14}p_{23}=0
\end{equation} are in the image of the embedding.

To relate the (un)stable two-vector with the (un)stable bundle, observe that \begin{equation}\label{basis of bundles 1}
\begin{aligned}
u_1(\lambda,z)\wedge u_2(\lambda,z)\\
u_3(\lambda,z)\wedge u_4(\lambda,z)
\end{aligned}\hspace{.1 in} \text{ solve } (\ref{derivation on ext alg}).
\end{equation} Furthermore, we compute from (\ref{individual soln decay}) that \begin{equation}\label{basis of bundles 2}
\begin{aligned}
\lim\limits_{z\rightarrow\infty}e^{-(\mu_1(\lambda)+\mu_2(\lambda))z}u_1(\lambda,z)\wedge u_2(\lambda,z)=\eta_1(\lambda)\wedge\eta_2(\lambda)\\
\lim\limits_{z\rightarrow-\infty}e^{-(\mu_3(\lambda)+\mu_4(\lambda))z}u_3(\lambda,z)\wedge u_4(\lambda,z)=\eta_3(\lambda)\wedge\eta_4(\lambda).
\end{aligned}
\end{equation} The two-vectors $\eta_1(\lambda)\wedge\eta_2(\lambda)$ and $\eta_3(\lambda)\wedge\eta_4(\lambda)$ are necessarily nonzero, since eigenvectors for different eigenvalues are linearly independent. On the other hand, we know from (\ref{soln decay}) that the only two-vectors satisfying (\ref{basis of bundles 1}) and (\ref{basis of bundles 2}) are proportional to $\tilde{E}^{s/u}(\lambda,z)$. It follows that for any $z\in\bbR$ and $\lambda\in I$, \begin{equation}
\begin{aligned}
j(E^s(\lambda,z)) & =[\tilde{E}^s(\lambda,z)]\\
j(E^u(\lambda,z)) & =[\tilde{E}^u(\lambda,z)],
\end{aligned}\end{equation} where $[\cdot]$ denotes the equivalence class in projective coordinates. This correspondence is important, since making use of the symplectic structure requires being able to distinguish different solutions in $E^{s/u}(\lambda,z)$. On that note, recall that \begin{equation}
\vp'(z)=(\hat{u}',\hat{v}',\hat{u}''/\sigma,\hat{v}''/\alpha)\in E^s(0,z)\cap E^u(0,z)
\end{equation} for all $z\in\bbR$ due to translation invariance, and the above intersection is one-dimensional by (A4). According to (A3), the decay rates of $\vp'$ are given by $\mu_2(0)$ in forwards time and $\mu_3(0)$ in backwards time. By Lemma 2.2 of \cite{GZ98}, we are free to take $\vp'(z)$ as the basis vector $u_2(0,z)$ (resp. $u_3(0,z)$) of $E^s(0,z)$ (resp. $E^u(0,z)$). We can then choose $u_{1/4}(0,z)$ so that the limits in (\ref{basis of bundles 2}) are exactly $\zeta_{s/u}(0)$, with the orientation given by (\ref{orientation}). In other words, we have \begin{equation}\label{large lambda individual}
\eta_1(0)\wedge\eta_2(0)\wedge\eta_3(0)\wedge\eta_4(0)=\zeta_s(0)\wedge\zeta_u(0)=\rho\,\mathrm{vol}^*.
\end{equation}This choice propagates to bases of $E^{s/u}(\lambda,z)$ (and hence to $\tilde{E}^{s/u}(\lambda,z)$ through the Pl\"{u}cker map) which agree with (\ref{soln decay}).

In what follows, we freely identify planes with their images under $j$. Before proceeding, it will be helpful to write down the matrix for $A^{(2)}$ using Pl\"{u}cker coordinates. Let $Y_1=(p,q,r,s)$ and $Y_2=(\tilde{p},\tilde{q},\tilde{r},\tilde{s})$ be solutions of (\ref{eval eqn 1st order}), such that $j(\text{sp}\{Y_1,Y_2\})=(p_{ij}).$ Then, using (\ref{eval eqn 1st order}) and (\ref{Plucker coord defn}), \[p_{12}'=(p\tilde{q}-\tilde{p}q)'=\sigma r\tilde{q}+\alpha p\tilde{s}-\sigma \tilde{r}q-\alpha \tilde{p}s=-\sigma(q\tilde{r}-r\tilde{q})+\alpha(p\tilde{s}-\tilde{p}s)=-\sigma p_{23}+\alpha p_{14}.\] The other derivatives are computed similarly, and we end up with \begin{equation}\label{Plucker matrix}
\left(\begin{array}{c}
p_{12}\\
p_{13}\\
p_{14}\\
p_{23}\\
p_{24}\\
p_{34}
\end{array}\right)'=
\left(\begin{array}{cccccc} 0&0&\alpha&-\sigma&0&0
\\ \noalign{\medskip}1&-c&0&0&0&0\\ \noalign{\medskip}\frac{\lambda-g'(\hat{v})}{\alpha} & 0 & -c & 0 & 0 & \sigma\\ \noalign{\medskip}\frac{f'(\hat{u})-\lambda}{\sigma}&0&0&-c&0&-{\alpha}
\\ \noalign{\medskip}1 & 0 & 0 & 0& -c & 0\\ \noalign{\medskip}0
& 1 & \frac{\lambda-f'(\hat{u})}{\sigma}& \frac{g'(\hat{v})-\lambda}{\alpha} & 1 & -2\,c \end{array}
 \right)\left(\begin{array}{c}
 p_{12}\\
 p_{13}\\
 p_{14}\\
 p_{23}\\
 p_{24}\\
 p_{34}
 \end{array}\right).
\end{equation}Now, one can verify directly that \begin{equation}
\frac{d}{dz}\left(p_{12}p_{34}-p_{13}p_{24}+p_{14}p_{23} \right)=0
\end{equation} for any nonzero solution $Y_1(\lambda,z)\wedge Y_2(\lambda,z)$ of (\ref{derivation on ext alg}). This is to be expected, since the flow of a linear ODE preserves the dimension of subspaces. Nonetheless, this suggests that a good way to find invariant subsets of $\text{Gr}_2(\bbR^4)$ for (\ref{derivation on ext alg}) is to find algebraic quantities that are preserved. For example, we have the following.  \begin{lemma}\label{Lag condition lma}
The set of planes satisfying $p_{13}-p_{24}=0$ is invariant under the flow of (\ref{derivation on ext alg}).
\end{lemma}

\begin{proof}
A direct computation gives \begin{equation}\label{omega deriv calc}
\begin{aligned}
(p_{13}-p_{24})' & = p_{12}-cp_{13}-(p_{12}-cp_{24})\\
& = -c(p_{13}-p_{24}).
\end{aligned}
\end{equation} The result follows from the uniqueness of solutions to ODEs.
\end{proof}
It is worth noting that equations like (\ref{Grassmannian condition}) and (\ref{omega deriv calc}) make sense because the polynomials are homogeneous. They therefore define projective varieties, which are identified with subsets of $\text{Gr}_2(\bbR^4)$ by $j$. The degree of the polynomial in Lemma \ref{Lag condition lma} is one, and it is called a linear line complex, see \S 1.11 of \cite{Duis}. It turns out that the subsets of $\text{Gr}_2(\bbR^4)$ cut out by a linear line complex have one of two different structures. To distinguish between them, it is useful to talk about differential forms.

Consider the dual space to $\bigwedge^2\bbR^4$, identified with $\bigwedge^2\bbR^4$ in the standard way. Elements of this space can be identified with two-forms on $\bbR^4$, see \cite{vinberg}, page 315. Explicitly, the element dual to $e_i\wedge e_j\in\bigwedge^2\bbR^4$ is identified with $de_i\wedge de_j$, which acts on two vectors in $\bbR^4$. We can therefore think of the Pl\"{u}cker embedding as identifying two-planes in $\bbR^4$ with differential two-forms (again, defined up to a nonzero multiple). With this interpretation, when $\lambda=0$, $\frac{d}{dz}$ becomes the Lie derivative of the form along the vector field tangent to the homoclinic orbit.

Since differential two-forms are skew-symmetric, any (non-trivial) two-form on $\bbR^4$ has either a two- or zero-dimensional kernel. In light of Lemma \ref{Lag condition lma}, we define \begin{equation}\label{omega defn}
\omega=de_1\wedge de_3-de_2\wedge de_4.
\end{equation} It is not difficult to see that $\omega$ is nondegenerate, hence it is a symplectic form. We can also describe the action of $\omega$ in terms of complex structures, see \cite{Duis}. Letting $\langle\cdot,\cdot\rangle$ denote the standard dot product on $\bbR^4$, we have \begin{equation}\label{omega defn complex structure}
\omega(a,b)=\langle a,Jb\rangle,
\end{equation} for any $a,b\in\bbR^4$. Here, $J$ is the complex structure \begin{equation}\label{J}
J=\left(\begin{array}{c c c c}
0 & 0 & 1 & 0\\
0 & 0 & 0 & -1\\
-1 & 0 & 0 & 0\\
0 & 1 & 0 & 0
\end{array}\right).
\end{equation} Formula (\ref{omega defn complex structure}) gives us an alternative way of proving Lemma \ref{Lag condition lma}. Let $u,v$ be linearly independent solutions of (\ref{eval eqn 1st order}) for fixed $\lambda\in\bbR$. Then \begin{equation}\label{Lemma pf 2}
\begin{aligned}
\frac{d}{dz}\omega(u,v) & = \omega(u_z,v)+\omega(u,v_z)\\
& = \omega(A(\lambda,z)u,v)+\omega(u,A(\lambda,z)v)\\
& = \langle A(\lambda,z)u,Jv\rangle+\langle u, JA(\lambda,z)v\rangle\\
& = \langle u,[A(\lambda,z)^TJ+JA(\lambda,z)]v\rangle.
\end{aligned}
\end{equation} The result would then follow from (\ref{omega defn complex structure}) if we could show that $A^TJ+JA=-cJ$. A simple calculation gives \begin{equation}
\begin{aligned}
A(\lambda,z)^TJ+JA(\lambda,z) & =
 \left(\begin{array}{c c c c}
-\frac{\lambda}{\sigma}+\frac{f'(\hat{u})}{\sigma} & -1 & 0 & 0\\
-1 & \frac{\lambda}{\alpha}-\frac{g'(\hat{v})}{\alpha} & 0 & 0\\
c & 0 & \sigma & 0\\
0 & -c & 0 & -\alpha
\end{array}\right)\\ & +\left(\begin{array}{c c c c}
\frac{\lambda}{\sigma}-\frac{f'(\hat{u})}{\sigma} & 1 & -c & 0\\
1 & -\frac{\lambda}{\alpha}+\frac{g'(\hat{v})}{\alpha} & 0 & c\\
0 & 0 & -\sigma & 0\\
0 & 0 & 0 & \alpha
\end{array}\right)\\
&=\left(\begin{array}{c c c c}
0 & 0 & -c & 0\\
0 & 0 & 0 & c\\
c & 0 & 0 & 0\\
0 & -c & 0 & 0
\end{array}\right)=-cJ,
\end{aligned}
\end{equation}
as desired. Now, using (\ref{omega defn complex structure}), it is not difficult to see that for any plane $V=(p_{ij})\in\text{Gr}_2(\bbR^4)$, we have \begin{equation}
p_{13}-p_{24}=0\iff \omega\vert_{V}=0.
\end{equation} It follows that the set $p_{13}-p_{24}=0$--which is invariant for the equation induced on $\Gr_2(\bbR^4)$ by (\ref{eval eqn 1st order})--is the set of Lagrangian planes for $\omega.$ A subspace $V\in\text{Gr}_2(\bbR^4)$ is called \emph{Lagrangian} if $\omega(v_1,v_2)=0$ for all $v_i\in V$. The set of Lagrangian planes is actually a differentiable manifold of dimension 3, called the Lagrangian Grassmannian $\Lambda(2)$, and it is a homogeneous space $U(2)/O(2)$. (See \cite{Arnold67}, for example.) It can be shown that $\pi_1(\Lambda(2))=\bbZ$, which allows us to define a winding number for curves in this space. This winding number is called the Maslov index, which will be discussed in \S 5. First, we prove the crucial result that the stable and unstable bundles are Lagrangian.
\begin{theorem}\label{un-stable bundles Lagrangian}
Both $E^s(\lambda,z)$ and $E^u(\lambda,z)$ are Lagrangian subspaces for all $\lambda$ and $z$.
\end{theorem}
\begin{proof}
Define another symplectic form $\Omega$ by \begin{equation}\label{Omega defn}
\Omega(a,b)=e^{cz}\omega(a,b).
\end{equation} Clearly, $E^{u/s}(\lambda,z)$ is $\omega-$Lagrangian if and only if it is $\Omega-$Lagrangian for any $z$. Furthermore, $\Omega(u,v)$ is independent of $z$ for any solutions $u,v$ of (\ref{eval eqn 1st order}). To see this, we compute \begin{equation}
\frac{d}{dz}\Omega(u,v) = e^{cz}\left(c\omega(u,v)+\frac{d}{dz}\omega(u,v) \right)=0,
\end{equation} by (\ref{Lemma pf 2}). Now, for the stable bundle, we see that \[\Omega(u_1,u_2)=\lim\limits_{z\rightarrow\infty}e^{cz}\omega(u_1(\lambda,z),u_2(\lambda,z))=0, \] since $c<0$ and $u_{1,2}$ decay exponentially in forward time for all $\lambda$. For the unstable bundle, notice that $\mu_3+\mu_4>-c$ for all $\lambda$, so \[e^{cz}\omega(u_3,u_4)=e^{(c+\mu_3+\mu_4)z}\omega(e^{-\mu_3z}u_3,e^{-\mu_4z}u_4)\rightarrow 0 \text{ as } z\rightarrow-\infty, \] since the exponential in front has a positive exponent, and the arguments of $\omega$ are bounded. In both cases, $\Omega$ is identically zero, so $\omega$ must be as well.
\end{proof}

\section{Symplectic Version of the Evans Function}
Recall that for $\lambda\in I$, we picked out spanning solutions $u_i(\lambda,z)$ for the stable and unstable bundles. This allows us to rewrite the Evans function \begin{equation}\label{Evans fct individual vecs}
\begin{aligned}
D(\lambda)\mathrm{vol}^*=e^{2cz}E^{s}(\lambda,z)\wedge E^{u}(\lambda,z) & =e^{2cz}u_1(\lambda,z)\wedge u_2(\lambda,z)\wedge u_3(\lambda,z)\wedge u_4(\lambda,z)\\
& = e^{2cz}\det[u_1,u_2,u_3,u_4]\,\mathrm{vol}^*.
\end{aligned}
\end{equation}

In order to exploit the symplectic structure of the preceding section in the Evans' function analysis, we need to rewrite $D(\lambda)$ in terms of the symplectic form. This idea was pioneered in \cite{BD99,BD01} for systems of Hamiltonian PDEs with a multi-symplectic structure, and the following formula first appeared in \cite{CB14}. The slight difference in our formula and that of Chardard-Bridges' is due to the fact that the symplectic form is different in activator-inhibitor systems.
\begin{theorem}\label{symplectic evans fct}
Let $a_1,a_2,b_1,b_2\in\bbR^4$. Then \begin{equation}\label{symplectic det}
\det[a_1,a_2,b_1,b_2]=-\det\left[\begin{array}{c c}
\omega(a_1,b_1) & \omega(a_1,b_2)\\
\omega(a_2,b_1) & \omega(a_2,b_2)
\end{array}\right]+\omega(a_1,a_2)\omega(b_1,b_2).
\end{equation}
\end{theorem}
This formula is proved for arbitrary (even) dimension in \cite{CB14} using the Leibniz formula for determinants. However, in this low-dimensional case it is easy enough to verify using brute force. Notice that the second term in (\ref{symplectic det}) disappears if either $\text{sp}\{a_1,a_2\}$ or $\text{sp}\{b_1,b_2\}$ is a Lagrangian plane. Combining (\ref{Evans fct individual vecs}) and (\ref{symplectic det}), we arrive at the symplectic Evans function. \begin{cor}\label{Evans fct def cor} The symplectic Evans function is \begin{equation}\label{sympelctic Evans fct}
D(\lambda)=-e^{2cz}\det\left[\begin{array}{c c}
\omega(u_1,u_3) & \omega(u_1,u_4)\\
\omega(u_2,u_3) & \omega(u_2,u_4)
\end{array}\right].
\end{equation}
\end{cor} In this form, it is easy to see the $z-$independence of $D$. Distributing one copy of $e^{cz}$ to each row of the matrix in (\ref{sympelctic Evans fct}), one can replace $e^{cz}\omega$ with $\Omega$ in each entry.

Let us now consider the case $\lambda=0$. First, due to translation invariance, the derivative of the traveling wave $\vp'(z)$ is a zero-eigenfunction for $L$, hence $D(0)$ should be zero. Indeed, following (A2) and the discussion in \S 3 we set $\vp'=u_2(0,z)=u_3(0,z)$. From Theorem \ref{un-stable bundles Lagrangian}, it follows that each entry of the matrix in (\ref{sympelctic Evans fct}) is zero, with the possible exception of $\omega(u_1,u_4)$. Thus $D(0)=0$, as expected. Corollary \ref{Evans fct def cor} can also be used to show that $D'(0)=0$ if the stable and unstable bundles have a two-dimensional intersection (i.e. they are tangent to each other). In this case, the matrix in (\ref{sympelctic Evans fct}) is the zero matrix for $\lambda=0$, and an application of the product rule shows that $D'(0)=0$.

The main result of this paper involves $D'(0)$, so we start by calculating this using Jacobi's formula. The second part of this calculation is inspired by the proof of Theorem 1.11 in \cite{PW}, and similar calculations are carried out in \cite{BD01}.
\begin{lemma}\label{Lemma D prime}
The quantity $D'(0)$ is given by \begin{equation}\label{Evans fct deriv formula}
D'(0)=\Omega(u_1,u_4)\int\limits_{-\infty}^{\infty}e^{cz}\left(\frac{(\hat{u}')^2}{\sigma}-\frac{(\hat{v}')^2}{\alpha}\right)\,dz.
\end{equation}
\end{lemma}
Before giving the proof, a few comments are in order. First, the fact that the wave is transversely constructed implies that both terms in the product in (\ref{Evans fct deriv formula}) are nonzero. The integral depends solely on the wave itself, and this will be calculable. The term $\Omega(u_1,u_4)$ on the other hand--named the Lazutkin-Treschev invariant in \cite{CB14}--carries information about the intersection of the stable and unstable manifolds. The decay rates of the $u_i$ and the $z-$independence of $\Omega$ can be used to show that $e^{-\mu_4(0)z}u_4(0,z)$ converges to a multiple of $\eta_4(0)$ in forward time. However, the orientation of this vector (i.e. whether that multiple is positive or negative) is difficult to ascertain, and it is what determines the sign of $D'(0)$. We will see that the Maslov index can be used to circumvent this difficulty. Now for the proof of the Lemma.

\begin{proof}
Denote by $\Sigma(\lambda,z)$ the matrix in (\ref{sympelctic Evans fct}), and let $\Sigma(\lambda,z)^\#$ be its adjugate (i.e. the transpose of its cofactor matrix). By the Jacobi formula (\S 8.3 of \cite{MN88}), we have \begin{equation}
\begin{aligned}
D'(0) & = -e^{2cz}\,\text{Trace}(\Sigma^\#\Sigma_\lambda)\vert_{\lambda=0}.\\
& =-\text{Trace}\left(\left[\begin{array}{c c}
\Omega(u_2,u_4) & -\Omega(u_1,u_4)\\
-\Omega(u_2,u_3) & \Omega(u_1,u_3)
\end{array}\right]\left[\begin{array}{c c}
\p_\lambda\Omega(u_1,u_2) & \p_\lambda\Omega(u_1,u_4)\\
\p_\lambda\Omega(u_2,u_3) & \p_\lambda\Omega(u_2,u_4)
\end{array}\right]\right)\big\vert_{\lambda=0}\\
& =-\text{Trace}\left(\left[\begin{array}{c c}
0 & -\Omega(u_1,u_4)\\
0 & 0
\end{array}\right]\left[\begin{array}{c c}
\p_\lambda\Omega(u_1,u_2) & \p_\lambda\Omega(u_1,u_4)\\
\p_\lambda\Omega(u_2,u_3) & \p_\lambda\Omega(u_2,u_4)
\end{array}\right]\right)\big\vert_{\lambda=0}\\
& = \Omega(u_1,u_4)\p_\lambda\Omega(u_2,u_3)\vert_{\lambda=0}.
\end{aligned}
\end{equation} The vanishing terms in the third equality are due to the fact that $u_2=u_3=\vp'$ when $\lambda=0$, and $E^{u/s}(0,z)$ are both Lagrangian planes for all $z$. It remains to calculate $\p_\lambda\Omega(u_2,u_3)$. Since $\mu_2(\lambda)+\mu_3(\lambda)\equiv -c$, we can write $\Omega(u_2,u_3)=\omega(U,V)$, where $U=e^{-\mu_2(\lambda)z}u_2$ and $V=e^{-\mu_3(\lambda)z}u_3$, from which it follows that $\omega(U,V)$ is $z-$independent, and $\p_\lambda\,\Omega(u_2,u_3)=\p_\lambda\,\omega(U,V)$. Furthermore, $U$ and $V$ satisfy the equations \begin{equation}\label{scaled soln z derivs}
\begin{aligned}
U_z & =(A(\lambda,z)-\mu_2(\lambda))\,U\\
V_z & = (A(\lambda,z)-\mu_3(\lambda))\,V.
\end{aligned}
\end{equation} Taking derivatives in $\lambda$, we have that \begin{equation}\label{solution lambda derivs}
\begin{aligned}
U_{\lambda z} & =(A(\lambda,z)-\mu_2(\lambda))\,U_\lambda+(A_\lambda-\mu_2'(\lambda))U\\
V_{\lambda z} & = (A(\lambda,z)-\mu_3(\lambda))\,V_\lambda+(A_\lambda-\mu_3'(\lambda))V,
\end{aligned}
\end{equation} where \begin{equation}\label{A lambda}
A_\lambda=\left(\begin{array}{c c c c}
0 & 0 & 0 & 0\\
0 & 0 & 0 & 0\\
\frac{1}{\sigma} & 0 & 0 & 0\\
0 & \frac{1}{\alpha} & 0 & 0
\end{array}\right)
\end{equation} is independent of $z$. Finally, since $\p_z\,\omega(U,V)=0$, $\p_z\p_\lambda\,\omega(U,V)=0$ as well, so \begin{equation}\label{product rule eq}
\p_z\,\omega(U_\lambda,V)=-\p_z\,\omega(U,V_\lambda).
\end{equation} Using (\ref{solution lambda derivs}) and (\ref{scaled soln z derivs}), we calculate that \begin{equation}
\begin{aligned}
\p_z\,\omega(U_\lambda,V) & =\omega\left(A(\lambda,z)U_\lambda-\mu_2(\lambda)U_\lambda+A_\lambda U-\mu_2'(\lambda)U,V\right)\\ & +\omega(U_\lambda,(A(\lambda,z)-\mu_3(\lambda))\,V)\\
& = -c\,\omega(U_\lambda,V)-(\mu_2(\lambda)+\mu_3(\lambda))\omega(U_\lambda,V)+\omega(A_\lambda U,V)-\mu_2'(\lambda)\,\omega(U,V)\\
& = \omega(A_\lambda U,V)-\mu_2'(\lambda)\,\omega(U,V).
\end{aligned}
\end{equation} In the second equality, we used the fact from (\ref{Lemma pf 2}) that \begin{equation}
\omega(A(\lambda,z)v_1,v_2)+\omega(v_1,A(\lambda,z)v_2)=-c\,\omega(v_1,v_2)
\end{equation} for any $z,\lambda$ and any vectors $v_i\in\bbR^4$. This yields the third equality in conjunction with the identity $\mu_2(\lambda)+\mu_3(\lambda)\equiv -c$. If we evaluate this expression at $\lambda=0$, whence $U=e^{-\mu_2(0)z}\vp'(z)$, $V=e^{-\mu_3(0)z}\vp'(z)$, and $\omega(U,V)=\Omega(\vp',\vp')=0$, we end up with \begin{equation}\label{z-lambda partial}
\p_z\,\omega(U_\lambda,V)(0,z)=e^{cz}\omega(A_\lambda\vp',\vp')=-e^{cz}\left(\frac{(\hat{u}')^2}{\sigma}-\frac{(\hat{v}')^2}{\alpha} \right).
\end{equation} To complete the proof, we use the Fundamental Theorem of Calculus and (\ref{product rule eq}), \`{a} la \cite{PW}. For any large $R,S>0$, we have \begin{equation}
\begin{aligned}
\omega(U_\lambda,V)(0,0) & =\omega(U_\lambda,V)(0,R)+\int\limits_{0}^{R}e^{cz}\left(\frac{(\hat{u}')^2}{\sigma}-\frac{(\hat{v}')^2}{\alpha} \right)dz\\
\omega(U,V_\lambda)(0,0) & =\omega(U,V_\lambda)(0,-S)+\int\limits_{-S}^{0}e^{cz}\left(\frac{(\hat{u}')^2}{\sigma}-\frac{(\hat{v}')^2}{\alpha} \right)dz
\end{aligned}
\end{equation} Adding these equations and taking $R,S\rightarrow\infty$ gives the desired result, provided that the boundary terms vanish in the limit. Since the limits in (\ref{individual soln decay}) are achieved uniformly on compact subsets of $I$, we know that the limits \begin{equation}
\begin{aligned}
\lim\limits_{z\rightarrow-\infty}U_\lambda(\lambda,z)\\
\lim\limits_{z\rightarrow\infty}V_\lambda(\lambda,z)
\end{aligned}
\end{equation} exist. Furthermore, for $\lambda =0$, it is clear that $V=e^{-\mu_3(0)z}\vp'\rightarrow 0$ as $z\rightarrow\infty$ and that $U=e^{-\mu_2(0)z}\rightarrow 0$ as $z\rightarrow-\infty$, so the boundary terms vanish by the linearity of $\omega$, giving the result.
\end{proof} Having derived an expression for $D'(0)$, the task is now to determine its sign. To do so, we must understand the term $\Omega(u_1,u_4)$ from (\ref{Evans fct deriv formula}). The key to doing so is the Maslov index.

\section{The Maslov Index and Detection Form}
For any $\lambda\in I$, we know that the assignment $z\mapsto E^u(\lambda,z)$ yields a curve in $\text{Gr}_2(\bbR^4)$. Following the discussion in \S 3 (specifically Theorem \ref{un-stable bundles Lagrangian}), it actually defines a curve in $\Lambda(2)$, the space of \emph{Lagrangian} planes in $\bbR^4$. To any such curve, we can assign an integer invariant called the Maslov index. In \cite{CB14}, the Evans function is related to the Maslov index using a formulation of the latter due to Souriau \cite{Souriau}. This is defined for two elements in the universal cover of $\Lambda(2)$. Here, we opt for the more typical definition of the Maslov index, using intersections with the train of a fixed subspace. This approach to defining the Maslov index for homoclinic orbits has its origins in \cite{BJ,CH}. It was also used in \cite{JLM13} for periodic boundary value problems.

We now review the intersection definition of the Maslov index. This was introduced in \cite{Arnold67}, but we will need the improvement of \cite{RS93}, since there is necessarily an intersection at the right endpoint of the curve in our case due to translation invariance. Let $V\in\Lambda(2)$ be a fixed Lagrangian subspace. We can write $\Lambda(2)$ as the disjoint union of sets $\Sigma_k(V)\subset\Lambda(2),$ $k=0,1,2$, where $k$ is the dimension of the intersection between $V$ and any plane in $\Sigma_k(V)$. Each $\Sigma_k(V)$ is a submanifold of $\Lambda(2)$ of codimension $k(k+1)/2$. In particular, $\Sigma_2(V)$ contains only the plane $V$ itself. It is not difficult to see that $\overline{\Sigma_1(V)}=\Sigma_1(V)\cup\Sigma_2(V)$. The plane $V$ is called the \emph{reference plane}, and $\Sigma(V)$ the \emph{train} of $V$. We will sometimes write $\Sigma$ instead of $\Sigma(V)$ if the reference plane is unambiguous. In Arnol'd's work \cite{Arnold67,Ar85}, this set is called the \emph{singular cycle}.

The Maslov index is a count of how many times a curve $\gamma(z)\in\Lambda(2)$ intersects the train $\Sigma$ of some fixed reference plane $V\in\Lambda(2)$. The intersections are weighted by a ``crossing form,'' which was introduced in \cite{RS93} and will be discussed in more detail below. One of the assumptions we will make is that our curve has only regular crossings with the singular cycle. When the intersection is one-dimensional, this simply means that the crossing is transverse. For our purposes, the natural curve to consider is the unstable bundle $z\mapsto E^u(0,z)$. The reference plane is taken to be $V=E^s(0,\tau)$, where $\tau$ is large enough so that $V^u(0)\cap E^s(0,\tau')=\{0\}$ for all $\tau'\geq\tau$. The reasons for choosing this reference plane will be explained later. We now elaborate on the crossing form and define the Maslov index.
\begin{define}
$z^*\in\bbR$ is called a \textbf{crossing} (or conjugate point) if $E^u(0,z^*)\cap E^s(0,\tau)\neq \{0\}$. At a $k$-dimensional crossing $z^*$ $(k=1,2)$, there is a quadratic form \begin{equation}\label{crossing form}
\Gamma(E^u,E^s(0,\tau);z^*)(\xi)=\omega(\xi,A(0,z^*)\xi),
\end{equation} defined on the intersection $E^u(0,z^*)\cap E^s(0,\tau)$. $\Gamma$ is called the \textbf{crossing form}. A crossing $z^*$ is called \textbf{regular} if $\Gamma$ is nonsingular. If, in addition, $k=1$, the crossing is called \textbf{simple}. 
\end{define}
It is implicit in the preceding definition that $\Gamma$ is derived from a symmetric bilinear form. Indeed, let $z^*$ be a conjugate time, and consider the bilinear form $(v_1,v_2)\mapsto\omega(v_1,A(0,z^*)v_2)$ defined on $E^u(0,z^*)\cap E^s(0,\tau)$. Using (\ref{omega deriv calc}) and the fact that $E^u(0,z)$ is Lagrangian for all $z$, we see that \begin{equation}
\omega(v_1,A(0,z^*)v_2)-\omega(v_2,A(0,z^*)v_1)=\frac{d}{dz}\omega(v_1,v_2)=-c\omega(v_1,v_2)=0,
\end{equation} as desired. In order to define the Maslov index, it still needs to be shown that this crossing form is equivalent to the one developed in Theorem 1.1 of \cite{RS93}. This is the content of the following theorem. \begin{theorem}
The crossing form $\Gamma$ in (\ref{crossing form}) is well-defined. In other words, the Maslov index defined by this crossing form is equivalent to the Maslov index in \cite{RS93}, and consequently it enjoys the same properties.
\end{theorem}

\begin{proof}
Suppose that $z^*$ is a conjugate point. Take $W\in\Lambda(2)$ such that $E^u(0,z^*)\oplus W=\bbR^4$. It is well known (e.g. \S 1.6 of \cite{Duis}) that any other Lagrangian subspace transverse to $W$ can be written uniquely as the graph of a linear operator $A:E^u(0,z^*)\rightarrow W$. In particular, for $|z-z^*|<\delta\ll 1$, \begin{equation}\label{complementary subspace}
E^u(0,z)=\{v+\psi(z)v:v\in E^u(0,z^*)\}\end{equation} with $\psi(z):E^u(0,z^*)\rightarrow W$ smooth in $z$. For any $\xi\in E^u(0,z^*)\cap E^s(0,\tau)$, we therefore have a curve $w(z)\in W$ defined by $\xi+w(z)\in E^u(0,z)$, or, equivalently, $w(z)=\psi(z)\xi$. Furthermore, we have $\psi(z^*)=0$. It is shown on page 3 of \cite{RS93} that the form \begin{equation}\label{R-S crossing form}
Q(\xi)=\frac{d}{dz}\big\vert_{z=z^*}\,\omega(\xi,w(z))
\end{equation} is independent of the choice of $W$ and defines the crossing form. It therefore suffices to show that we can recover (\ref{crossing form}) from (\ref{R-S crossing form}). To that end, it will be helpful to consider the evolution operator $\Phi(\zeta,z)$ for (\ref{eval eqn 1st order}) with $\lambda=0$. $\Phi$ satisfies $\Phi(\zeta,\zeta)=\text{Id}$ and $\Phi(z^*,z)\cdot E^u(0,z^*)=E^u(0,z)$. (Here, $\cdot$ refers to the induced action of $\Phi(z^*,z)$ on a two-dimensional subspace.) Notice that (\ref{complementary subspace}) defines a curve $\gamma(z)\in E^u(0,z^*)$ by the formula \begin{equation}
\xi+\psi(z)\xi=\Phi(z^*,z)\gamma(z).
\end{equation} From above, $\gamma(z^*)=\xi$. We are now ready to compute: \begin{equation}
\begin{aligned}
\frac{d}{dz}\big\vert_{z=z^*}\,\omega(\xi,w(z)) & = \frac{d}{dz}\big\vert_{z=z^*}\,\omega(\xi,\Phi(z^*,z)\gamma(z)-\xi)\\
& = \frac{d}{dz}\big\vert_{z=z^*}\,\omega\left(\xi,\Phi(z^*,z)\gamma(z)\right)\\
& = \omega\left(\xi,A(0,z^*)\gamma(z^*)\right)+\omega\left(\xi,\gamma'(z^*) \right)\\
& = \omega\left(\xi,A(0,z^*)\xi\right)+\lim\limits_{z\rightarrow z^*}\frac{1}{z-z^*}\omega\left(\xi,\gamma(z)-\gamma(z^*) \right)\\
& = \omega\left(\xi,A(0,z^*)\xi\right).
\end{aligned}
\end{equation} The last equality follows since $\gamma(z)\in E^u(z^*)$ for all $z$, which is a Lagrangian plane containing $\xi$. This completes the proof.
\end{proof}
We are now able to define the Maslov index of $\varphi$. We stress that the index depends on the choice of reference plane. For a quadratic form $Q$, we define $n_+(Q)$ and $n_-(Q)$ to be respectively the positive and negative indices of inertia of $Q$ (page 187 of \cite{vinberg}). Thus \begin{equation}
\mathrm{sign}(Q)=n_+(Q)-n_-(Q).
\end{equation}
\begin{define}\label{Maslov of phi defn num}
Let $\tau\gg 1$ be as above. The \textbf{Maslov index} of $\varphi$ is given by \begin{equation}\label{Maslov defn}
\mathrm{Maslov}(\varphi):=\sum_{z^*\in(-\infty,\tau)}\mathrm{sign}\,\Gamma(E^u,E^s(0,\tau),z^*)+n_+(\Gamma(E^u,E^s(0,\tau),\tau)),
\end{equation} where the sum is taken over all interior crossings of $E^u(0,z)$ with $\Sigma$, the train of $E^s(0,\tau)$.
\end{define}
\begin{rem}
The term $n_+(\Gamma(E^u,E^s(0,\tau),\tau))$ appears in the above definition because $\tau$ is an endpoint crossing. As explained in \cite{RS93}, care must be taken with such crossings to ensure that the Maslov index is additive with respect to concatenation of curves (Theorem 2.3 of \cite{RS93}). In fact, our convention differs from that used in \cite{RS93}, in which $(1/2)\mathrm{sign}\Gamma$ is assigned to each endpoint of the curve. We instead use the convention of \cite{BCJLMS17,HLS16}, which is to assign $-n_-(Q)$ to crossings at a left endpoint and $n_+(Q)$ to crossings at a right endpoint. In so doing, we ensure that the Maslov index is an integer as opposed to a half-integer.
\end{rem}
The natural choice of reference plane would seem to be $V^s(0)$. This is problematic, however, because we know that $E^u(0,z)$ approaches the train of $V^s(0)$ as $z\rightarrow\infty$; $\vp'$ spans the intersection in the limit. The crossing form would have to approach 0, so it would be impossible to determine the sign of this final crossing. The idea of pulling back $V^s(0)$ slightly to $E^s(0,\tau)$ is due to Chen and Hu, see \cite{CH}. Using the properties of the Maslov index derived in \S 2 of \cite{RS93}, they prove that the Maslov index given by (\ref{Maslov defn}) is independent of $\tau$. This strategy forces an intersection at the right endpoint $z=\tau$, since $\varphi'(\tau)\in E^u(0,\tau)\cap E^s(0,\tau)$. Assumption (A4) guarantees that this intersection is one-dimensional, so it suffices to evaluate the crossing form $\Gamma$ on $\varphi'$ at $z=\tau$. If the crossing is in the positive direction, then $+1$ is contributed to the Maslov index. If the crossing is negative, then there is no contribution to the Maslov index, by (\ref{Maslov defn}).

Finally, we remind the reader that it is assumed all crossings with the train are regular. Indeed, this assumption is necessary if one wishes to define the Maslov index as a homotopy invariant for non-closed curves \cite{RS93}. However, this is not a practical concern for $\mathrm{Maslov}(\varphi)$, since we have control over the reference plane in the form of $\tau$. The image of the curve in $\Lambda(2)$ is independent of $\tau$ (other than where it ends), but by varying $\tau$ one could move the singular cycle $\Sigma$ (which is codimension one in $\Lambda(2)$) to break any tangential (i.e. irregular) crossings, since these are non-generic (c.f. \S 2.1 of \cite{Arnold67}).

The rest of this section is dedicated to proving the formula \begin{equation}\label{parity result}
(-1)^{\text{Maslov}({\vp})}=\text{sign}\,\Omega(u_1,u_4).
\end{equation} We will do this through an analysis of the conjugate point $\tau$, which is the right endpoint of the curve $(-\infty,\tau]\rightarrow E^u(0,z)\subset\Lambda(2)$. This point is critical because it encodes the translation invariance which necessitates that $\lambda=0$ be an eigenvalue of $L$. It is therefore not surprising that it should be the distinguished $z$ value that is used to connect the Evans function and Maslov index. Now, to calculate the Maslov index, we must have a way of finding the other conjugate points. This is accomplished through the introduction of the \emph{detection form} $\pi\in\left(\bigwedge^2\bbR^{4}\right)^*,$ defined by  \begin{equation}\label{detection form defn}
\pi(w_1\wedge w_2)= \det\left[e^{-\mu_1(0)\tau}u_1(0,\tau),e^{-\mu_2(0)\tau}u_2(0,\tau),w_1,w_2\right].
\end{equation} $\pi$ is called the detection form because it is 0 precisely when the plane $W=\text{sp}\{w_1,w_2\}$ intersects $E^s(0,\tau)$ nontrivially. Thus it detects conjugate points for a curve of Lagrangian planes. This form is traditionally called the dual to the characterizing two-vector $w_1\wedge w_2$ for $W$, see pp. 97-98 of \cite{PirCra}. We next define a function $\beta:\bbR\rightarrow\bbR$, which evaluates $\pi$ on $E^u(0,z)$. Explicitly, we have \begin{equation}\label{beta defn}
\beta(z)= e^{-(\mu_1+\mu_2)\tau-(\mu_3+\mu_4)z}\det[u_1(\tau),u_2(\tau),u_3(z),u_4(z)].
\end{equation} We henceforth suppress the dependence of $u_i,\mu_i$ on $\lambda$, since we take $\lambda=0$ for this calculation. For brevity, we also set $M(z)=-(\mu_1+\mu_2)\tau-(\mu_3+\mu_4)z$. Recall that $u_2=u_3=\vp'$ for $\lambda=0$, so we see immediately that $\beta(\tau)=0$, since columns two and three are both $\vp'(\tau)$. 

Now, we can use (\ref{symplectic det}) to rewrite $\beta$ as \begin{equation}\label{beta defn symp}
\beta(z)=-e^{M(z)}\det\left[\begin{array}{c c}
\omega(u_1(\tau),u_3(z)) & \omega(u_1(\tau),u_4(z))\\
\omega(u_2(\tau),u_3(z)) & \omega(u_2(\tau),u_4(z))
\end{array}\right].
\end{equation}
The next ingredient is $\beta'(\tau)$, whose sign we claim will help determine the sign of $D'(0)$. Since $\beta(\tau)=0$, we see that \begin{equation}
\beta'(\tau) = -e^{M(\tau)}\frac{d}{dz}\left[\omega(u_1(\tau),u_3(z))\omega(u_2(\tau),u_4(z))-\omega(u_1(\tau),u_4(z))\omega(u_2(\tau),u_3(z))\right]\big\vert_{z=\tau}
\end{equation}
Before jumping into the product rule expansion, recall that $u_2(\tau)=u_{3}(\tau)=\vp'(\tau)$, and hence $\text{sp}\{u_i(\tau),u_j(\tau)\}$ is Lagrangian for $(i,j)=(1,2),(1,3),(2,4),(3,4)$, with $\omega(u_2,u_3)=0$ as well. It follows that the only surviving term is $-\omega(u_1(\tau),u_4(\tau))\omega(u_2(\tau),u_3'(\tau))$. Since $M(\tau)=-(\mu_1+\mu_2+\mu_3+\mu_4)\tau=2c\tau$, we conclude that \begin{equation}\label{beta prime tau}
\beta'(\tau)=\Omega(u_1,u_4)\Omega\left(\vp'(\tau),\vp''(\tau)\right).
\end{equation}
The relation to (\ref{Evans fct deriv formula}) is now apparent. Noticing that $\vp''=A(0,z)\vp'$, the second term in (\ref{beta prime tau}) is the crossing form for the conjugate point $z=\tau$, scaled by a positive factor $e^{c\tau}$. We will show that the sign of $\Omega(u_1,u_4)$ can be determined from the Maslov index, regardless of the sign of the crossing at $z=\tau$. The tie that binds the two is $\beta(z)$. First, from (\ref{beta defn}) we can see that $\beta$ is asymptotically constant as $z\rightarrow-\infty$. Indeed, if $\tau$ is large enough, then (\ref{orientation}) and (\ref{large lambda individual}) imply that \begin{equation}\label{beta minus inf}
\lim\limits_{z\rightarrow-\infty}\beta(z) \approx\det\left[\eta_1,\eta_2,\eta_3,\eta_4\right]=\rho>0.
\end{equation} Thus $\beta(z)>0$ for large, negative $z$, provided $\tau$ is large enough. By definition, zeros of $\beta$ correspond to conjugate points for the curve $E^u(0,z)$. Heuristically, the sign of $\beta'(\tau)$ is positive if there are an odd number of conjugate points (excluding $\tau$) and negative if there are an even number of conjugate points. Since the Maslov index, roughly speaking, counts the number of conjugate points, its parity should therefore determine the sign of $\beta'(\tau)$.

To make the preceding precise, we must know a few things about zeros of $\beta$ and the Maslov index. First, since an intersection of $E^u(0,z)$ with $E^s(0,\tau)$ can be one- or two-dimensional, the contribution to the Maslov index at any (interior) conjugate point is $-2,-1,0,1,$ or $2$. Since the parity of the index is unchanged if the contribution is even, we need $\beta$ to cross through the $z-$axis if and only if the crossing is one-dimensional. Obviously, we also need $\beta$ to have finitely many zeros for this to make sense. The latter is true if we assume that there are only regular crossings, which is an assumption needed to define the Maslov index for non-loops in the first place, see \cite{RS93}. It turns out that the assumption of regularity is also sufficient for the former. This is the content of the next two lemmas.

\begin{lemma}\label{1d crossing lemma}
If $z^*$ is a conjugate point such that the intersection $E^u(0,z^*)\cap E^s(0,\tau)=\mathrm{sp}\{\xi \}$ is one-dimensional, then the crossing is regular if and only if $\beta'(z^*)\neq 0$.
\end{lemma}

\begin{proof}
Let $\xi=\beta_1u_1(\tau)+\beta_2u_2(\tau)$ be a vector in the intersection. Let $\nu$ be a second basis vector for $E^u(0,z^*)$. As noted in \S 4 of \cite{CB14}, $\omega(u_i(\tau),\nu)\neq0$ $(i=1,2)$, else we would have $\nu\in E^s(0,z^*)$, violating the assumption that the intersection is one-dimensional. Changing to the basis $\{\xi,\nu\}$ of $E^u(0,z^*)$ would introduce a nonzero multiple in the expression for $\beta(\tau)$, which we call $B$. It follows that \begin{equation}\label{1d crossing}
\begin{aligned}
\beta'(z^*) & =Be^{M(z^*)}\left\{\det[u_1(\tau),u_2(\tau),A(0,z^*)\xi,\nu]+\det[u_1(\tau),u_2(\tau),\xi,A(0,z^*)\nu] \right\}\\
& = Be^{M(z^*)}\det[u_1(\tau),u_2(\tau),A(0,z^*)\xi,\nu]\\
& = -Be^{M(z^*)}\det\left[\begin{array}{c c}
\omega(u_1(\tau),A(0,z^*)\xi) & \omega(u_1(\tau),\nu)\\
\omega(u_2(\tau),A(0,z^*)\xi) & \omega(u_2(\tau),\nu)
\end{array}\right].\\
\end{aligned}
\end{equation} Since $\text{sp}\{\xi,\nu\}$ is a Lagrangian plane, we have \begin{equation}
0=\omega(\xi,\nu)=\beta_1\omega(u_1(\tau),\nu)+\beta_2\omega(u_2(\tau),\nu).
\end{equation} Without loss of generality, we can assume $\beta_2\neq 0$, and hence $\omega(u_2(\tau),\nu)=-\displaystyle\frac{\beta_1}{\beta_2}\omega(u_1(\tau),\nu).$ Returning to (\ref{1d crossing}), we see that \begin{equation}
\begin{aligned}
\beta'(z^*) & = -Be^{M(z^*)}\left\{-\frac{\beta_1}{\beta_2}\omega(u_1(\tau),A(0,z^*)\xi)\omega(u_1(\tau),\nu)-\omega(u_2(\tau),A(0,z^*)\xi)\omega(u_1(\tau),\nu) \right\}\\
& = \frac{B}{\beta_2}e^{M(z^*)}\omega(u_1(\tau),\nu)\omega(\beta_1u_1(\tau)+\beta_2u_2(\tau),A(0,z^*)\xi)\\
& = \frac{B}{\beta_2}\omega(u_1(\tau),\nu)e^{M(z^*)}\omega(\xi,A(0,z^*)\xi).
\end{aligned}
\end{equation} Comparing with (\ref{crossing form}), it's now clear that the crossing is regular if and only if $\beta'(z^*)\neq 0$.
\end{proof}

\begin{lemma}\label{2d crossing lemma}
If $z^*$ is a conjugate point such that the intersection $E^u(0,z^*)\cap E^s(0,\tau)$ is two-dimensional, then the following are true:
\begin{enumerate}
\item $\beta'(z^*)=0$. 
\item $\beta''(z^*)\neq0\iff$ the crossing at $\tau$ is regular.
\end{enumerate}
\end{lemma}

\begin{proof} 
Immediately we see that \begin{equation}\label{2d crossing 1st deriv}
\begin{aligned}
\beta'(z^*) & = e^{M(z^*)}\left\{\det[u_1(\tau),u_2(\tau),A(0,z^*)u_3(z^*),u_4(z^*)]+\det[u_1(\tau),u_2(\tau),u_3(z^*),A(0,z^*)u_4(z^*)] \right\}\\
& =0,
\end{aligned}
\end{equation} since there is linear dependence in the first, second and fourth (resp. first, second and third) columns in the matrix on the left (resp. right). In a similar way, the second derivative is seen to be \begin{equation}\label{2d crossing 2nd deriv}
\beta''(z^*)  =2e^{M(z^*)}\det[u_1(\tau),u_2(\tau),A(0,z^*)u_3(z^*),A(0,z^*)u_4(z^*)]
\end{equation} Next, since the crossing is two-dimensional, we have $\text{sp}\{u_1(\tau),u_2(\tau) \}=\text{sp}\{u_3(z^*),u_4(z^*) \},$ so by some change of basis in the first two columns of (\ref{2d crossing 2nd deriv}), we end up with \begin{equation}\label{2d crossing 2nd deriv 2}
\begin{aligned}
\beta''(z^*) & =2Be^{M(z^*)}\det\left[u_3(z^*),u_4(z^*),A(0,z^*)u_3(z^*),A(0,z^*)u_4(z^*)\right]\\
& = -2Be^{M(z^*)}\det\left[\begin{array}{c c}
\omega(u_3(z^*),A(0,z^*)u_3(z^*)) & \omega(u_3(z^*),A(0,z^*)u_4(z^*))\\
\omega(u_4(z^*),A(0,z^*)u_3(z^*)) & \omega(u_4(z^*),A(0,z^*)u_4(z^*))
\end{array} \right],
\end{aligned}
\end{equation} using (\ref{symplectic det}). The symplectic version of the matrix in (\ref{2d crossing 2nd deriv 2}) is exactly the matrix of the crossing form $\Gamma$ in the basis $\{u_3(z^*),u_4(z^*) \}$ for $E^s(0,\tau)\cap E^u(0,z^*)$. To say that the crossing is regular then is to say that this matrix does not have zero as an eigenvalue. Since the determinant of this matrix is the product of the eigenvalues, the Lemma follows. 
\end{proof}
These lemmas allow us to conclude the following: consider the curve $\gamma(z)$, which is $E^u(0,z)$ restricted to an interval $(-\infty,\tau-\eps)$ containing all conjugate points prior to $\tau$. Then \begin{equation}
\mu(\gamma,E^s(0,\tau))\text{ is even}\iff \beta'(\tau)<0.
\end{equation} In other words, since $\beta(\tau)=0$, the direction in which $\beta(z)$ crosses through 0 at $\tau$ is completely determined by how many times $\beta(z)$ passed through the $z-$axis prior to $\tau$.

To calculate $\mathrm{Maslov}(\varphi)$, one would need to know the direction of the final crossing, i.e. the sign of $\omega(\vp'(\tau),\vp''(\tau)).$ However, this is \emph{not} needed to prove (\ref{parity result}). First, assume that ${\text{Maslov}({\vp})}$ is even. There are now two possibilities regarding the final crossing at $z=\tau$. If $\omega(\vp'(\tau),\vp''(\tau))>0$, then this crossing contributes $+1$ to the index, which means that there were an odd number of weighted crossings prior to $\tau$: odd $+1=$ even. (In the above notation, $\mu(\gamma,E^s(0,\tau))$ is odd.) Thus $\beta'(\tau)>0$, from which we conclude that $\Omega(u_1,u_4)>0$, using (\ref{beta prime tau}). 

On the other hand, if $\omega(\vp'(\tau),\vp''(\tau))<0$, then there must be an even number of weighted crossings prior to $\tau$, since the last crossing contributes $0$ to the count (being a negative crossing). This implies that $\beta'(\tau)<0$. Again using (\ref{beta prime tau}), we see that $\Omega(u_1,u_4)>0$, showing that its sign does not depend on the direction of the final crossing. An analogous argument shows that $\Omega(u_1,u_4)<0$ if and only if ${\text{Maslov}({\vp})}$ is odd. This proves the formula (\ref{parity result}).

\section{Application}
The Maslov index could be defined in an analogous fashion to Definition \ref{Maslov of phi defn num} for any $\lambda\in\bbR$. However, the value $\lambda=0$ plays a distinguished role, since the eigenvalue equation in that case corresponds to the variational equation for (\ref{traveling wave ODE}) along $\varphi$. One then sees that the unstable bundle $E^u(0,z)$ is tangent to $W^u(0)$ at each point along the wave. (An analogous statement holds for $E^s(0,z)$ and $W^s(0)$.) This is observed on page 73 of \cite{AJ94} and page 196 of \cite{Yan02}, but the argument is straightforward enough to outline here. Indeed, a tangent vector to $W^u(0)$ at any point of $\varphi(z)$ can be associated to a one-parameter family $\varphi_s(z)$ of orbits in $W^u(0)$, where $\varphi_0(z)=\varphi(z)$. If this family is parametrized by $s$, then $Y(z):=\partial_s\varphi_s(z)|_{s=0}$--the $s$ derivative of $\varphi_s(z)$ evaluated along $\varphi$--is seen to be a solution of the variational equation. Furthermore, the exponential decay in backwards time of all trajectories in $W^u(0)$ guarantees that $Y(z)$ converges exponentially to $0$ as well, hence it is in $E^u(0,z)$.

The upshot of the preceding paragraph is that the Maslov index can be calculated using information from the nonlinear system (\ref{traveling wave ODE}). This can sometimes be more accessible than analyzing (\ref{eval eqn 1st order}) directly. For example, Fenichel theory \cite{Fen79,JoGSP} is an invaluable tool for tracking invariant manifolds in singularly perturbed systems. These ideas were exploited in \cite{Jones84} to prove that fast waves for the Fitz-Hugh Nagumo system are stable. In that case, the sign of $D'(0)$ was known by Evans \cite{Evans4} to be related to how the center-stable and center-unstable manifolds of the fixed point $0$ cross; importantly, this was done by augmenting $c'=0$ to the underlying traveling wave equation (hence the center direction). Similarly, derivatives of $D(\lambda)$ with respect to various parameters were used to derive instability criteria in \cite{AJ94,BD99,BD01,PW}, to name a few. The Maslov index analysis herein differs from those just mentioned in that the speed parameter is fixed. The relevant information about the twisting of the unstable bundle is contained in the four-dimensional phase space of (\ref{traveling wave ODE}).

To expand on this, we consider an example. The paradigmatic activator-inhibitor system is the aforementioned FitzHugh-Nagumo equation \begin{equation}\label{FHN 4d}
\begin{aligned}
u_t & = u_{xx}+f(u)-v\\
v_t & = v_{xx}+\eps(u-\gamma v),
\end{aligned}
\end{equation} where $f(u)=u(1-u)(u-a)$. Typically, $\eps$ is taken to be very small, making this a singular perturbation problem. Furthermore, $a$ satisfies $0<a<1/2$. The stability of various traveling and standing fronts and pulses has been studied for the variation of (\ref{FHN 4d}) in which there is either no diffusion on $v$, or the diffusion coefficient is a small parameter, for example \cite{Jones84,Flo91,AGJ,Yana89}. For the case of equal diffusivities, a stability result for standing waves was obtained in \cite{CH14}.

It was shown recently \cite{CC15} using variational techniques that (\ref{FHN 4d}) possesses fast (i.e. speed $c=O(1)$ in $\eps$) traveling pulses when $\gamma$ in (\ref{FHN 4d}) is chosen small enough so that the only fixed point of the kinetics equation \begin{equation}\label{FHN kinetics}
\left(\begin{array}{c}
u\\v
\end{array}\right)'=\left(\begin{array}{c}
f(u)-v\\
\eps(u-\gamma v)
\end{array}\right)
\end{equation} is $(u,v)=(0,0)$. In another work \cite{CJ17}, we offer an existence proof for these waves based on geometric singular perturbation theory \cite{JoGSP}. Furthermore, we prove that these waves are stable using the Maslov index. As suggested above, the calculation of the index is aided by the timescale separation inherent in (\ref{FHN 4d}). The full calculation is lengthy, but we are able to give a taste of it below. We stress again that the Maslov index is fundamentally different than the orientation index used in \cite{JoGSP,AJ94}, since we do not vary $c$ in order to get the result.

Written as a first-order system, the traveling wave equation for (\ref{FHN 4d}) is \begin{equation}\label{FHN traveling wave ODE}
\left(\begin{array}{c}
u\\
v\\
w\\
y
\end{array}\right)_z=\left(\begin{array}{c}
w\\
\eps y\\
-cw-f(u)+v\\
-cy+\gamma v-u
\end{array}\right).
\end{equation}
This is a fast-slow system with three fast variables $(u,w,y)$ and one slow variable $v$. For an overview of fast-slow systems, we refer the reader to \cite{JoGSP,Kuehn}.

There is a one-dimensional critical manifold $M_0$ given by \begin{equation}
M_0=\{(u,v,w,y):v=f(u),w=0,y=\frac{1}{c}(\gamma v-u)\},
\end{equation}which is normally hyperbolic wherever $f'(u)\neq 0$. The limiting slow flow on $M_0$ is given by \begin{equation}\label{slow flow}
\dot{v}=y=\frac{1}{c}(\gamma v-f^{-1}(v)),
\end{equation} where $f^{-1}$ is defined separately on three segments of the cubic $v=f(u)$, divided by the two zeros of $f'(u)$. Of particular interest are the two outer branches corresponding to the intervals on which $f(u)$ is strictly decreasing. We call these $M_0^L$ and $M_0^R$ for the left and right branches respectively. As in the case with no diffusion on $v$, it can be shown that there is a singular homoclinic orbit $\vp_0$ to $(0,0,0,0)$ consisting of alternating fast and slow pieces. Specifically, there is a value of $c<0$ such that a heteroclinic connection exists from $(u,v,w,y)=(0,0,0,0)$ on $M_0^L$ to $(u,v,w,y)=(1,0,0,-1/c)$ on $M_0^R$. From there, the slow flow carries us up $M_0^R$ to a point $q=(f^{-1}(v^*),v^*,0,(1/c)(\gamma v^*-f^{-1}(v^*)))$ at which another heteroclinic connection exists back to $M_0^L$. After making this fast jump back to $M_0^L$, the orbit is closed up by the slow flow returning to $(0,0,0,0)$. Using the Exchange Lemma \cite{JoGSP,JK94}, this singular orbit can be shown to perturb to a homoclinic orbit $\vp_\eps$ of (\ref{FHN traveling wave ODE}) for $0<\eps\ll 1$. This, in turn, corresponds to a traveling pulse for (\ref{FHN 4d}). The construction of the pulse is very similar to that of the analogous pulses for (\ref{FHN 4d}) with no diffusion on $v$. More details of that construction can be found in \cite{JKL91}, or \S 2 of \cite{Jones84}.

The main task in \cite{CJ17} is to use the fast-slow decomposition to calculate the Maslov index. More specifically, we use the timescale separation to follow the two-dimensional $W^u(0)$ around phase space. This is done by considering separately the fast and slow segments, as well as the transitions between them. For example, consider $E^u(0,z)=T_{\vp_\eps(z)}W^u(0)$ along the fast jump. The direction of the slow flow on $M_0^L$ indicates that both unstable directions for the fixed point $0$ are fast. As long as any conjugate points are regular, then the intersections with the train $\Sigma(E^s(0,\tau))$ will persist to the $\eps\neq 0$ case. We can therefore follow $W^u(0)$ along the fast jump from $M_0^L$ to $M_0^R$ in the reduced system \begin{equation}\label{fast subsystem}
\left(\begin{array}{c}
u\\
w\\
y
\end{array}\right)_z=\left(\begin{array}{c}
w\\
-cw-f(u)\\
-cy+\gamma v-u
\end{array}\right).
\end{equation}
 It turns out that the unstable manifold along this jump is easy to describe. Indeed, observe that the equations for $u$ and $w$ decouple from $y$. It follows that any solution of (\ref{fast subsystem}) must project onto a solution of \begin{equation}\label{Nagumo front}
\left(\begin{array}{c}
u\\w
\end{array}\right)'=\left(\begin{array}{c}
w\\
-cw-f(u)
\end{array}\right).
\end{equation} This is the steady state equation for the traveling fronts considered in \cite{FMc77}. For this system, it is known \cite{McKean70} that there is a heteroclinic orbit connecting $(0,0)$ and $(1,0)$ for the value \begin{equation}\label{FHN speed}
c=c^*=\sqrt{2}(a-1/2)<0.
\end{equation}
The profile of this solution is given by \begin{equation}\label{fast jump profile}
 w(u)=\frac{\sqrt{2}}{2}u(1-u), \hspace{.2 in}0\leq u\leq 1.
 \end{equation} To get the full picture of $W^u(0)$, note that the linearization about $0$ in (\ref{fast subsystem}) is given by \begin{equation}\label{lin about 0 fast jump}
\left(\begin{array}{c}
\delta u\\
\delta w\\
\delta y
\end{array}\right)'=\left(\begin{array}{c c c}
0 & 1 & 0\\
a & -c & 0\\
-1 & 0 & -c
\end{array}\right)\left(\begin{array}{c}
\delta u\\
\delta w\\
\delta y
\end{array}\right)
\end{equation}

One sees that $(0, 0, 1)^T$ is an eigenvector for the matrix in (\ref{lin about 0 fast jump}) with eigenvalue $-c>0$. Thus the second unstable direction is the invariant $y$ direction. It follows that $W^u(0)$ for $0\leq u\leq 1$ is a cylinder over the Nagumo front (\ref{fast jump profile}). A heteroclinic orbit from $(0,0,0)$ to $(1,0,-1/c)$ is then shown to exist by a shooting argument in this cylinder. By embedding (\ref{fast subsystem}) in $\bbR^4$ with $v=0$ as a parameter, we see that there is a heteroclinic connection $\gamma(z)$ from $(0,0,0,0)$ to $p=(1,0,0,-1/c)$, and this orbit is $O(\eps)$ close to $\vp_\eps$, up to when the landing point $p$ on $M_0^R$ is reached. Furthermore, in light of (\ref{fast jump profile}), we have \begin{equation}\label{cyl tangent space}
E^u(0,z)\approx T_{\gamma}W^u(0)=\mathrm{sp}\left\{\left[\begin{array}{c}
1\\
0\\
\sqrt{2}/2-u\sqrt{2}\\
0
\end{array}\right],\left[\begin{array}{c}
0\\0\\0\\1
\end{array}\right] \right\}.
\end{equation} Here, the notation $E^u(0,z)$ is used to remind the reader that this is the curve of interest in $\Lambda(2)$. Exact values of $z$ are meaningless in this limit, since it takes infinite time for the fast jump. However, the Maslov index is independent of parametrization; it is only the image of the curve in $\Lambda(2)$ that matters, so this analysis still detects any conjugate points that are encountered on the fast jump.

Recalling Definition \ref{Maslov of phi defn num} and the ensuing paragraph, we know that the reference plane should be $V^s(0)$--the stable subspace of the fixed point $0$--flown backwards slightly along the wave. The leading order approximation to this plane is known from Fenichel theory \cite{Fen79,JoGSP}. This theory says that $M_0^{R/L}$ and their respective stable and unstable manifolds perturb to corresponding objects when $\eps>0$ is small. Moreover, the foliation of (e.g.) $W^s(M_0^L)$ given by the individual stable eigenvectors along $M_0^L$ is preserved as well, see \S 3.3 of \cite{JoGSP}. We can therefore take our reference plane to be \begin{equation}\label{ref plane fhn}
V=\mathrm{sp}\left\{\left[\begin{array}{c}
1/f'(u_\tau)\\
1\\
0\\
(1/c^*)(\gamma-1/f'(u_\tau))
\end{array} \right],\left[\begin{array}{c}
f'(u_\tau)\\
0\\
f'(u_\tau)\mu_1(u_\tau)\\
\mu_1(u_\tau)
\end{array} \right] \right\},
\end{equation} where $u_\tau$ is the $u$-coordinate of the point close to $0$ on $M_\eps^L$ at which $E^s(0,\tau)$ is pinned. The first vector in (\ref{ref plane fhn}) is obtained by differentiating the equations defining $M_0^L$ with respect to $v$, and the second is the stable eigenvector from (\ref{lin about 0 fast jump}), with corresponding eigenvalue $\mu_1(u_\tau)$. One may object that the first basis vector is a center (not stable) direction when $\eps=0$. However, it becomes stable when $\eps>0$, and the limit of this subspace as $\eps\rightarrow 0$ is smooth. It is therefore the correct space to consider in the limit.

By definition, a conjugate point $z^*$ is a value of $z$ such that $E^u(0,z^*)\cap E^s(0,\tau)\neq\{0\}$. These spaces are $O(\eps)$ close to $T_\gamma W^u(0)$ and $V$ respectively, so we detect conjugate points by seeing for which values of $u$ (the variable parameterizing $\gamma$) $T_\gamma W^u(0)\cap V\neq\{0\}$. From (\ref{cyl tangent space}) and (\ref{ref plane fhn}), it is clear such an intersection occurs only when \begin{equation}
u=\frac{1}{2}-\frac{\mu_1(u_\tau)}{\sqrt{2}}.
\end{equation}
Since $u$ increases monotonically along the fast jump, there can be at most one conjugate point. A straightforward calculation from (\ref{lin about 0 fast jump}) and (\ref{FHN speed}) shows that there is, in fact, a conjugate point, since \begin{equation}\label{front eval bounds}
-a\sqrt{2}<\mu_1(u_\tau)<-\frac{\sqrt{2}}{2},
\end{equation} and the intersection is spanned by 
\begin{equation}
\xi=\left[\begin{array}{c}
f'(u_\tau)\\
0\\
f'(u_\tau)\mu_1(u_\tau)\\
\mu_1(u_\tau)
\end{array} \right].
\end{equation}

To calculate the contribution to the Maslov index, we evaluate the crossing form (\ref{crossing form}) on $\xi$. Let $u_P$ be the $u$-coordinate of $\vp_0$ at the conjugate point $z^*$. We then compute

\begin{equation}\label{fast front xing form}
\omega(\xi,A(0,z^*)\xi)=-(f'(u_\tau))^2(\mu_1(u_\tau)^2+c\mu_1(u_\tau)+f'(u_P)).
\end{equation} 
Using (\ref{FHN speed}), (\ref{front eval bounds}), and the value \begin{equation}
\mu_1(u_\tau)=-\frac{c}{2}-\sqrt{c^2-4f'(u_\tau)}, 
\end{equation} it can be shown that the expression in (\ref{fast front xing form}) is negative, thus the contribution to the Maslov index is $-1$ along the fast front.

After landing on the right slow manifold $M_0^R$, the next stage of the singular orbit is to flow up $M_0^R$ to the point $q$ where a second heteroclinic connection exists back to $M_0^L$. To determine the contribution to the Maslov index along this slow piece, we must understand what happens to $W^u(0)$. By Deng's Lemma \cite{Schec08}, we know that $W^u(0)$ will be crushed against $W^u(M_0^R)$, the unstable manifold of $M_0^R$. Since each point on $M_0^R$ has two unstable directions, $W^u(M_0^R)$ is a three-dimensional set. We must therefore determine which unstable direction is picked out, since $W^u(0)$ is only two-dimensional. Indeed, the underlying orbit itself is tangent (to leading order) to $M_0^R$, so the configuration of $W^u(0)$ is determined by this second direction, which must be unstable. It turns out that one can use the symplectic structure to conclude that the strong unstable direction persists, since we know from \S 2 that the tangent space to $W^u(0)$ must everywhere be a Lagrangian plane. Armed with this information, it is then possible to show that there is a unique conjugate point along $M_0^R$--it occurs when $\vp'$ is (to leading order) parallel to $T_{E^s(0,\tau)}M_0^L$. However, the crossing form calculation in this case indicates that the crossing is in the positive direction, hence the contribution to the Maslov index is $+1$.

One of the challenging aspects of the analysis is determining how the transition from fast to slow dynamics occurs on the level of tangent planes. Indeed, the curve $E^u(0,z)\subset \Lambda(2)$ is discontinuous in the singular limit, since the configuration of $W^u(0)$ is different upon entering and leaving the landing point $p$ on $M_0^R$. In \cite{CJ17} we prove that there is no contribution to the Maslov index near $p$, nor in the other two ``corners" at which fast-to-slow transitions occur. Moreover, in the manner indicated in this section, we show that there is a single conjugate point on each fast and slow piece for a total of four; the crossings along the fast jumps are in the negative direction, and the crossings along the slow pieces are in the positive direction. The last crossing occurs at $z=\tau$, which is the conjugate point guaranteed to exist by Definition \ref{Maslov of phi defn num}. Since this crossing is in the positive direction, it contributes $+1$ to the Maslov index (as opposed to 0, c.f. (\ref{Maslov defn})). It follows that \begin{equation}
\mathrm{Maslov}(\vp)=-1+1-1+1=0.
\end{equation}

We show in \cite{CJ17} how this information can be used to prove that the fast traveling waves are stable. Knowing this, it must therefore be the case that $D'(0)>0$, since we know that $D(\lambda)>0$ for $\lambda\gg 1$, and there are no positive eigenvalues for $L$. The calculation outlined above shows that $\mathrm{Maslov}(\vp)=0$ is even, so it must be the case that \begin{equation}\label{deriv inequality}
\int\limits_{-\infty}^{\infty}e^{cz}\left((\hat{u}')^2-\frac{(\hat{v}')^2}{\eps}\right)\,dz>0
\end{equation} as well, in light of Lemma \ref{Lemma D prime}. Sure enough, \S 2 of \cite{CC15} (particularly equation (2.7) and Lemma 2.1) can be adapted to show that the inequality (\ref{deriv inequality}) holds, and hence $D'(0)>0$, as expected.

\paragraph*{\textbf{Acknowledgments:}} This research was partially supported by National Science Foundation (NSF) Grant DMS-1312906 and by the Office of Naval Research (ONR) Grant N00014-15-1-2112.

\bibliographystyle{amsplain}
\bibliography{AI_Maslov_v2}

\providecommand{\bysame}{\leavevmode\hbox to3em{\hrulefill}\thinspace}
\providecommand{\MR}{\relax\ifhmode\unskip\space\fi MR }
\providecommand{\MRhref}[2]{%
  \href{http://www.ams.org/mathscinet-getitem?mr=#1}{#2}
}
\providecommand{\href}[2]{#2}
\begin{thebibliography}{10}

\bibitem{AGJ}
J.~Alexander, R.A. Gardner, and C.K.R.T. Jones, \emph{A topological invariant
  arising in the stability analysis of travelling waves}, J. reine angew. Math
  \textbf{410} (1990), no.~167-212, 143.

\bibitem{AJ94}
J.C. Alexander and C.K.R.T. Jones, \emph{Existence and stability of
  asymptotically oscillatory double pulses}, J. reine angew. Math \textbf{446}
  (1994), 49--79.

\bibitem{Arnold67}
Vladimir~Igorevich Arnol'd, \emph{Characteristic class entering in quantization
  conditions}, Functional Analysis and its applications \textbf{1} (1967),
  no.~1, 1--13.

\bibitem{Ar85}
\bysame, \emph{The {S}turm theorems and symplectic geometry}, Functional
  analysis and its applications \textbf{19} (1985), no.~4, 251--259.

\bibitem{BJ89}
Peter~W. Bates and Christopher~K.R.T. Jones, \emph{Invariant manifolds for
  semilinear partial differential equations}, Dynamics reported, Springer,
  1989, pp.~1--38.

\bibitem{BCJLMS17}
Margaret Beck, Graham Cox, Christopher Jones, Yuri Latushkin, Kelly McQuighan,
  and Alim Sukhtayev, \emph{Instability of pulses in gradient
  reaction-diffusion systems: A symplectic approach}, arXiv preprint
  arXiv:1705.03861 (2017).

\bibitem{BM15}
Margaret Beck and Simon Malham, \emph{Computing the {M}aslov index for large
  systems}, Proceedings of the American Mathematical Society \textbf{143}
  (2015), no.~5, 2159--2173.

\bibitem{BJ}
Amitabha Bose and Christopher~K.R.T. Jones, \emph{Stability of the in-phase
  travelling wave solution in a pair of coupled nerve fibers}, Indiana
  University Mathematics Journal \textbf{44} (1995), no.~1, 189--220.

\bibitem{BD99}
Thomas~J. Bridges and Gianne Derks, \emph{Unstable eigenvalues and the
  linearization about solitary waves and fronts with symmetry}, Proceedings of
  the Royal Society of London A: Mathematical, Physical and Engineering
  Sciences, vol. 455, The Royal Society, no. 1987, 1999, pp.~2427--2469.

\bibitem{BD01}
\bysame, \emph{The symplectic {E}vans matrix, and the instability of solitary
  waves and fronts}, Archive for rational mechanics and analysis \textbf{156}
  (2001), no.~1, 1--87.

\bibitem{BD03}
\bysame, \emph{Constructing the symplectic {E}vans matrix using maximally
  analytic individual vectors}, Proceedings of the Royal Society of Edinburgh:
  Section A Mathematics \textbf{133} (2003), no.~03, 505--526.

\bibitem{CB14}
Fr{\'e}d{\'e}ric Chardard and Thomas~J. Bridges, \emph{Transversality of
  homoclinic orbits, the {M}aslov index and the symplectic {E}vans function},
  Nonlinearity \textbf{28} (2014), no.~1, 77.

\bibitem{CDB09}
Fr{\'e}d{\'e}ric Chardard, Fr{\'e}d{\'e}ric Dias, and Thomas~J Bridges,
  \emph{Computing the {M}aslov index of solitary waves, part 1: {H}amiltonian
  systems on a four-dimensional phase space}, Physica D: Nonlinear Phenomena
  \textbf{238} (2009), no.~18, 1841--1867.

\bibitem{CC15}
Chao-Nien Chen and YS~Choi, \emph{Traveling pulse solutions to
  {F}itz{H}ugh--{N}agumo equations}, Calculus of Variations and Partial
  Differential Equations \textbf{54} (2015), no.~1, 1--45.

\bibitem{CH}
Chao-Nien Chen and Xijun Hu, \emph{Maslov index for homoclinic orbits of
  {H}amiltonian systems}, Annales de l'Institut Henri Poincar{\'e}, vol.~24,
  Analyse non Lin{\'e}are, no.~4, 2007, pp.~589--603.

\bibitem{CH14}
\bysame, \emph{Stability analysis for standing pulse solutions to
  {F}itz{H}ugh--{N}agumo equations}, Calculus of Variations and Partial
  Differential Equations \textbf{49} (2014), no.~1-2, 827--845.

\bibitem{CJ17}
P.~Cornwell and C.K.R.T Jones, \emph{On the existence and stability of fast
  traveling waves in a doubly--diffusive {F}itz{H}ugh--{N}agumo system}, In
  preparation.

\bibitem{CJM15}
Graham Cox, Christopher~KRT Jones, and Jeremy~L Marzuola, \emph{A {M}orse index
  theorem for elliptic operators on bounded domains}, Communications in Partial
  Differential Equations \textbf{40} (2015), no.~8, 1467--1497.

\bibitem{PirCra}
Mike Crampin and Felix~A.E. Pirani, \emph{Applicable differential geometry},
  vol.~59, Cambridge University Press, 1986.

\bibitem{Duis}
J.J. Duistermaat, \emph{Symplectic geometry}, Spring School, June 7-14, 2004.

\bibitem{Evans4}
John~W Evans, \emph{Nerve axon equations: Iv. {T}he stable and the unstable
  impulse}, Indiana University Mathematics Journal \textbf{24} (1975), no.~12,
  1169--1190.

\bibitem{Fen79}
Neil Fenichel, \emph{Geometric singular perturbation theory for ordinary
  differential equations}, Journal of Differential Equations \textbf{31}
  (1979), no.~1, 53--98.

\bibitem{FMc77}
Paul~C Fife and J~Bryce McLeod, \emph{The approach of solutions of nonlinear
  diffusion equations to travelling front solutions}, Archive for Rational
  Mechanics and Analysis \textbf{65} (1977), no.~4, 335--361.

\bibitem{Flo91}
Gilberto Flores, \emph{Stability analysis for the slow travelling pulse of the
  {F}itz{H}ugh--{N}agumo system}, SIAM journal on mathematical analysis
  \textbf{22} (1991), no.~2, 392--399.

\bibitem{GZ98}
Robert~A. Gardner and Kevin Zumbrun, \emph{The gap lemma and geometric criteria
  for instability of viscous shock profiles}, Communications on pure and
  applied mathematics \textbf{51} (1998), no.~7, 797--855.

\bibitem{Hassett}
Brendan Hassett, \emph{Introduction to algebraic geometry}, Cambridge
  University Press, 2007.

\bibitem{Henry}
Dan Henry, \emph{Geometric theory of semilinear parabolic equations}, Lecture
  notes in mathematics, Springer-Verlag, Berlin, New York, 1981.

\bibitem{HomSan10}
Ale~Jan Homburg and Bj{\"o}rn Sandstede, \emph{Homoclinic and heteroclinic
  bifurcations in vector fields}, Handbook of dynamical systems \textbf{3}
  (2010), 379--524.

\bibitem{HLS16}
Peter Howard, Yuri Latushkin, and Alim Sukhtayev, \emph{The {M}aslov and
  {M}orse indices for {S}chr{\"o}dinger operators on $\mathbb{R}$}, arXiv
  preprint arXiv:1608.05692 (2016).

\bibitem{HS16}
Peter Howard and Alim Sukhtayev, \emph{The {M}aslov and {M}orse indices for
  {S}chr{\"o}dinger operators on $[0, 1]$}, Journal of Differential Equations
  \textbf{260} (2016), no.~5, 4499--4549.

\bibitem{JLS17}
Christopher Jones, Yuri Latushkin, and Selim Sukhtaiev, \emph{Counting spectrum
  via the {M}aslov index for one dimensional $\theta$-periodic
  {S}chr{\"o}dinger operators}, Proceedings of the American Mathematical
  Society \textbf{145} (2017), no.~1, 363--377.

\bibitem{Jones84}
Christopher~K.R.T. Jones, \emph{Stability of the travelling wave solution of
  the {F}itz{H}ugh--{N}agumo system}, Transactions of the American Mathematical
  Society \textbf{286} (1984), no.~2, 431--469.

\bibitem{JoGSP}
Christopher~KRT Jones, \emph{Geometric singular perturbation theory}, Dynamical
  systems, Springer, 1995, pp.~44--118.

\bibitem{JLM13}
Christopher~K.R.T. Jones, Yuri Latushkin, and Robert Marangell, \emph{The
  {M}orse and {M}aslov indices for matrix {H}ill’s equations}, Spectral
  analysis, differential equations and mathematical physics: a festschrift in
  honor of Fritz Gesztesy’s 60th birthday \textbf{87} (2013), 205--233.

\bibitem{JK94}
CKRT Jones and N~Kopell, \emph{Tracking invariant manifolds with differential
  forms in singularly perturbed systems}, Journal of Differential Equations
  \textbf{108} (1994), no.~1, 64--88.

\bibitem{JKL91}
CKRT Jones, N~Kopell, and R~Langer, \emph{Construction of the
  {F}itz{H}ugh-{N}agumo pulse using differential forms}, Patterns and dynamics
  in reactive media \textbf{37} (1991), 101--116.

\bibitem{KP13}
Todd Kapitula and Keith Promislow, \emph{Spectral and dynamical stability of
  nonlinear waves}, vol. 457, Springer, 2013.

\bibitem{Kato}
Tosio Kato, \emph{Perturbation theory for linear operators}, Springer, 1966.

\bibitem{Kuehn}
Christian Kuehn, \emph{Multiple time scale dynamics}, vol. 191, Springer, 2015.

\bibitem{MN88}
Jan~R Magnus and Heinz Neudecker, \emph{Matrix differential calculus with
  applications in statistics and econometrics}, Wiley series in probability and
  mathematical statistics (1988).

\bibitem{McKean70}
Henry~P. McKean, \emph{Nagumo's equation}, Advances in mathematics \textbf{4}
  (1970), no.~3, 209--223.

\bibitem{Murray}
James~D. Murray, \emph{Mathematical biology. ii spatial models and biomedical
  applications}, Springer-Verlag New York Incorporated, 2001.

\bibitem{PW}
Robert~L. Pego and Michael~I. Weinstein, \emph{Eigenvalues, and instabilities
  of solitary waves}, Philosophical Transactions of the Royal Society of London
  A: Mathematical, Physical and Engineering Sciences \textbf{340} (1992),
  no.~1656, 47--94.

\bibitem{RS93}
Joel Robbin and Dietmar Salamon, \emph{The {M}aslov index for paths}, Topology
  \textbf{32} (1993), no.~4, 827--844.

\bibitem{Sandstede02}
Bj{\"o}rn Sandstede, \emph{Stability of travelling waves}, Handbook of
  dynamical systems \textbf{2} (2002), 983--1055.

\bibitem{Schec08}
Stephen Schecter, \emph{Exchange lemmas 1: {D}eng's lemma}, Journal of
  Differential Equations \textbf{245} (2008), no.~2, 392--410.

\bibitem{Souriau}
Jean-Marie Souriau, \emph{Construction explicite de l'indice de {M}aslov.
  {A}pplications}, Group theoretical methods in physics, Springer, 1976,
  pp.~117--148.

\bibitem{turing}
A.M. Turing, \emph{The chemical basis of morphogenesis}, Bulletin of
  mathematical biology \textbf{52} (1990), no.~1-2, 153--197.

\bibitem{vinberg}
{\.E}rnest~Borisovich Vinberg, \emph{A course in algebra}, no.~56, American
  Mathematical Soc., 2003.

\bibitem{Yana89}
E~Yanagida, \emph{Stability of travelling front solutions of the
  {F}itzhugh--{N}agumo equations}, Mathematical and Computer Modelling
  \textbf{12} (1989), no.~3, 289--301.

\bibitem{Yan02}
Eiji Yanagida, \emph{Standing pulse solutions in reaction-diffusion systems
  with skew-gradient structure}, Journal of Dynamics and Differential Equations
  \textbf{14} (2002), no.~1, 189--205.

\end{thebibliography}

\end{document}